\documentclass[11pt]{amsart}
\usepackage[colorlinks=true, pdfstartview=FitV,  linkcolor=blue, citecolor=Green,
urlcolor=cyan]{hyperref}
\usepackage[dvipsnames]{xcolor}
\usepackage{stmaryrd}
\usepackage{enumerate}
\usepackage[pdftex]{graphicx}
\usepackage{amsmath}
\usepackage{amssymb}
\usepackage{mathtools}
\usepackage{mathrsfs}
\usepackage{faktor}
\usepackage{mathdots}
\usepackage{amsthm}
\usepackage{relsize}
\usepackage{bbm}
\usepackage{color}
\usepackage{tikz}
\usetikzlibrary{math}
\usepackage{pgfplots}
\usepackage[scale=0.8, top=2cm, bottom=2cm, left=2cm, right=2cm, marginparwidth=1.5cm, marginparsep=.1cm]{geometry}
\usepackage{pgfplots}
\usepgfplotslibrary{fillbetween,decorations.softclip}
\pgfplotsset{compat = newest}

\theoremstyle{plain}
\newtheorem{proposition}{Proposition}[section]
\newtheorem{theorem}[proposition]{Theorem}
\newtheorem{lemma}[proposition]{Lemma}
\newtheorem{corollary}[proposition]{Corollary}
\theoremstyle{definition}

\newtheorem{definition}[proposition]{Definition}

\newtheorem{remark}[proposition]{Remark}
\newtheorem{notation}[proposition]{Notation}

\def\bb #1{\mathbb{#1}}
\def\bf #1{\mathbf{#1}}

\def\wt #1{\widetilde{#1}}

\newcommand{\arctanh}{\mathrm{arctanh}}
\newcommand{\dil}{\mathrm{Li}_2}

\newcommand{\PGL}{\mathsf{PGL}}
\newcommand{\R}{\mathbb R}
\newcommand{\vol}{\textnormal{vol}}
\newcommand{\HT}{\textnormal{HT}}
\newcommand{\st}{0}
\newcommand{\out}{\textnormal{out}}
\newcommand{\inn}{\textnormal{in}}

\newcounter{notes}

\numberwithin{equation}{section}

\title[Holmes--Thompson area of inscribed polygons]{Holmes--Thompson area of inscribed polygons \\ and convex projective structures 
}
\author{Xenia Flamm}
\address{Institut des Hautes \'Etudes Scientifiques\newline
\indent Max-Planck Institute for Mathematics in the Sciences}
\email{flamm@ihes.fr, xenia.flamm@mis.mpg.de}
\author{Giuseppe Martone}
\address{Sam Houston State University}
\email{gxm120@shsu.edu}

\begin{document}

\date{\today}

\begin{abstract}
Positive tuples of complete flags in $\mathbb{R}^3$ define two convex polygons in $\mathbb{RP}^2$, one inscribed in the other.
We are interested in relating the Holmes--Thompson area of the inner polygon for the Hilbert metric on the outer polygon to the double and triple ratios of the positive tuple of flags.
This article focuses on positive triples and quadruples of flags.
For quadruples,  we investigate the special cases of hyperbolic quadrilaterals and the parametrization of the finite area convex real projective structures on a thrice-punctured sphere.
\end{abstract}

\maketitle

\section{Introduction}

A (marked) convex projective structure on a surface $S$ is an identification of the universal cover of $S$ with an open \emph{properly convex} subset $\Omega \subset \mathbb{RP}^2$,  meaning that $\Omega$ is contained in some affine chart of $\mathbb{RP}^2$ in which $\Omega$ is convex and bounded.
Such subsets of projective space are naturally endowed with a metric: the Hilbert metric.
This metric descends to a metric on $S$, turning the surface into a \emph{Finsler manifold}.
Loosely speaking, a Finsler manifold is a manifold endowed with a continuous family of norms on each tangent space. The norm allows to define a metric on a Finsler manifold, as well as several (non-equivalent) notions of volume \cite{AlvarezThompson_VolumesNormedFinslerSpaces}.
In this article we are interested in the \emph{Holmes--Thompson volume}: Given a Finsler manifold $M$ and $X\subset M$, Holmes and Thompson \cite{HolmesThompson_NDimAreaContentMinkowskiSpaces} define a notion of volume of $X$ by using the natural symplectic structure on the cotangent bundle of $M$ to integrate unit balls, with respect to the dual Finsler norm, at the points of $X$. The goal of this paper is the study of the Holmes--Thompson area of (finite area) convex projective surfaces.
Other notions of area of convex projective surfaces are already studied in \cite{AdeboyeCooper_AreaConvexProjSurfacesFGCoord} for the $p$-area, and in \cite{ColboisVernicosVerovic,marquis2012surface, Wolpert_HilbertAreaInscribedTrianglesQuadrilaterals}  for the Busemann area, which coincides with the Hausdorff measure. Note that these three notions are bilpschitz equivalent by \cite[Proposition 9.4]{Marquis_around}. We believe that the Holmes--Thompson area, which satisfies a Crofton formula by \cite{AlvarezBerck_WhatIsWrongHausdorffMeasureFinslerSpaces}, should appear when computing self-intersections of Hilbert geodesic currents (see \cite{Kim_Papadopoulos, Labourie_cross, MZ}) of convex real projective structures.

In their seminal work, Fock--Goncharov \cite{FockGoncharov_ModuliSpacesLocalSystemsHigherTeichmuellerTheory}, see also \cite{BonahonDreyer_ParametrizingHitchinComponents, BonahonDreyer_characters, FockGoncharov_ModuliSpacesConvexProjectiveStructuresSurfaces}, parametrize the space of (marked) finite area convex projective structures on $S$, denoted $\mathcal{CP}(S)$, in terms of projective invariants of positive tuples of flags in $\mathbb R^3$.
A natural question is thus to express the area of a convex projective surface in terms of these parameters. This seems to be a very difficult task in general.

In the case that $S=S_{0,3}$ is a thrice-punctured sphere,  
we obtain an asymptotic growth result of the Holmes--Thompson area of a convex projective structure on $S_{0,3}$.
The space $\mathcal{CP}(S_{0,3})$ can be parametrized by two real numbers ${\bf d}$ and ${\bf t}$ (the logarithms of a {\em double} and a {\em triple} ratio, respectively), and we estimate the growth rate of the Holmes--Thompson area of the convex projective structure $\Psi({{\bf d},{\bf t}})$ corresponding to ${\bf d}$ and ${\bf t}$ in terms of ${\bf d}$ and ${\bf t}$, as the Euclidean norm $\|(\bf d,\bf t)\|$ goes to infinity.

\begin{theorem}
\label{thm: intro: S03}
Let $\vol^{\HT}(\Psi({{\bf d},{\bf t}}))$ denote the Holmes--Thompson area of the convex projective structure $\Psi({{\bf d},{\bf t}})$ on the thrice-punctured sphere. Then,
\[
\liminf_{\|(\bf d,\bf t)\|\to\infty}\frac{\vol^{\HT}(\Psi({{\bf d},{\bf t}}))}{\bf d^2+\bf t^2}>0.
\]
\end{theorem}

The result follows from a standard comparison result,  saying that if $\Omega \subset \Omega'$ then $\vol^\HT_{\Omega} \geq \vol^\HT_{\Omega'}$ (Remark~\ref{rem: volume decreasing}),  and the precise study of the area of ideal quadrilaterals inscribed in quadrilaterals in $\mathbb{RP}^2$ in Theorem~\ref{thm: intro estimate}.

To see this, we first need to explain in more detail the parametrization of the space of finite area convex projective structures, due to Fock--Goncharov \cite{FockGoncharov_ModuliSpacesConvexProjectiveStructuresSurfaces}.
For simplicity we restrict our attention here to the case $S=S_{0,3}$, but analogous statements hold more generally. Let $p_1,p_2,p_3$ denote the three punctures of $S_{0,3}$ and let $\mathcal T_{\textnormal{bal}}$ denote the balanced ideal triangulation of $S_{0,3}$ consisting of edges oriented from $p_i$ to $p_{i+1}$, with $i\in\mathbb Z/3\mathbb Z$. 
Now let $\rho$ be a finite area convex projective structure on $S_{0,3}$. The universal cover of $S_{0,3}$ is then identified with an open proper convex set $\Omega_\rho$ in the projective plane $\mathbb{RP}^2$ satisfying the property that every point in the boundary of $\Omega_\rho$ has a unique tangent line. Lift the balanced triangulation of $S_{0,3}$ to an (oriented) ideal triangulation of $\Omega_\rho$ and pick two adjacent ideal triangles with ideal vertices $\xi_1,\xi_2,\xi_3,\xi_4$ appearing in this cyclic order along the boundary of $\Omega_{\rho}$. The four points $\xi_1,\xi_2,\xi_3,\xi_4$ in $\partial\Omega_\rho$ form a convex quadrilateral (the convex hull of these four points in the affine chart containing the closure of $\Omega_\rho$),  that is contained in the convex quadrilateral singled out by the four projective lines tangent to $\partial\Omega_\rho$ at $\xi_1,\xi_2,\xi_3,\xi_4$. We are thus naturally lead to consider pairs $\mathbf P=(P_\inn,P_\out)$ of inscribed convex quadrilaterals, and to estimate the Holmes--Thompson area of $P_\inn$ seen as a subset of $P_\out$. We also mention here that Faifman, Vernicos and Walsh \cite{FaifmanVernicosWalsh, VernicosWalsh} studied the asymptotic volume of the Hilbert geometry of a convex polytope,  and the relation to its combinatorial structure.

\subsection{Inscribed quadrilaterals}
To prove Theorem \ref{thm: intro: S03} we need a better understanding of the Holmes--Thompson area of inscribed quadrilaterals.
Since the Holmes--Thompson area is a projective invariant (Remark~\ref{rem: volume invariant}),  up to the action of $\mathsf{PGL}_3(\mathbb R)$,  we can normalize $P_\out$ to be equal to the {\em standard quadrilateral} $Q_\st$ with vertices $(\pm 1,\pm 1)$ and $(\pm 1,\mp 1)$. 
A key technical step, completed in Section~\ref{sec:HTQuadruple}, is finding a formula to compute the Holmes--Thompson area of subsets of $Q_\st$.

\begin{proposition}
\label{prop:intro:integrand:quadrilateral}
The dual unit ball for the Hilbert metric at a point $(x,y)$ in the interior of $Q_\st$ has Euclidean area $A_{Q_\st}$ equal to
\[
A_{Q_\st}(x,y) = \frac{2+\max\{|x|,|y|\}}{2(1-x^2)(1-y^2)}.
\]
In particular, the Holmes--Thompson area for any Lebesgue measurable subset $U \subset Q_\st$ is equal to 
\[ 
\vol^\HT_{Q_\st}(U) = \frac{1}{\pi}\int_U A_{Q_\st}(u)d\mu(u),
\]
where $d\mu$ is the standard Lebesgue measure on $\R^2$.
\end{proposition}

The general case of a pair of inscribed convex quadrilaterals $\mathbf P=(P_\inn,P_\out)$ presents significant challenges in concrete calculations.
Therefore we restrict our attention to two infinite families of examples of geometric interest in the study of convex real projective structures on surfaces.

\subsubsection{Hyperbolic quadrilaterals}
The first family arises from pairs of convex inscribed quadrilaterals $\mathbf P=(P_{\text{in}},P_{\text{out}})$ for which there exists an open ellipse $\mathcal E=\mathcal E(\textbf{P})$ such that $P_{\text{in}}$ is an ideal polygon in $\mathcal E$ and the boundary of $P_{\text{out}}$ is contained in the union of the tangent lines to $\mathcal E$ at the vertices of $P_{\text{in}}$.
Note that $\mathcal E$ equipped with its Hilbert metric is a (Klein--Beltrami) model of the hyperbolic plane and, for this reason, we say that the pair of inscribed quadrilaterals $\mathbf P$ is {\em hyperbolic}. 
In terms of Fock--Goncharov coordinates, for any choice of oriented diagonal for $P_{\text{in}}$ (see Remark \ref{rmk: combinatorics triple and double}),  these pairs are characterized by having both triple ratios equal to 1 and double ratios both equal to some $d\in (0,\infty)$, thus obtaining a one-parameter family of inscribed quadrilaterals which we denote $\bf P^{\textnormal{hyp}}(d)=(P_{\inn}^\textnormal{hyp}(d),P_{\out}^{\textnormal{hyp}}(d))$. The main result of Section ~\ref{sec:FuchsianLocus} is an explicit formula for the Holmes--Thompson area of $\bf P^{\textnormal{hyp}}(d)$ in terms of the double ratio $d$, which involves the {\em dilogarithm}.
\begin{definition}\label{def: dilogarithm}
For every $z\in\mathbb C$, the function $\displaystyle \dil(z)=-\int_0^z\frac{1}{v}\ln(1-v)dv$
is the {\em dilogarithm} of $z$.
\end{definition}
\begin{proposition}\label{prop: intro hyperbolic} For any $d>0$, let $\mathbf P^{\textnormal{hyp}}(d)=(P_{\inn}^\textnormal{hyp}(d),P_{\out}^{\textnormal{hyp}}(d))$ be a hyperbolic pair of inscribed convex quadrilaterals in $\mathbb {RP}^2$. 
Then, the Holmes--Thompson area of $P^{\textnormal{hyp}}_\inn(d)$ with respect to $P^{\textnormal{hyp}}_\out(d)$ is equal to
\begin{align*}
\vol^{\HT}&(\mathbf P^{\textnormal{hyp}}(d))=\frac{1}{2\pi}\Big(-\dil\left(-2d\right)
+2 \dil\left(\frac{-d}{1+d}\right)
-2\dil\left(-(1+d)\right)
-2\dil\left(\frac{d}{1+d}\right)\\
&-\dil\left(\frac{d}{2+d}\right)
+\dil\left(\frac{-d}{2+d}\right)
-2 \dil\left(-d\right)
-\dil\left(-(1+2d)\right)
-2\ln \left(1+d\right) \ln\left(\frac{2d}{1+d}\right)\\
&+\ln \left(\frac{2}{1+d}\right) 
\left(-3 \ln \left(\frac{2}{1+d}\right)
+\ln\left(\frac{2+4d}{1+d}\right)
+\ln (4)\right)
+2 \ln \left(\frac{4d}{1+d}\right)
\ln \left(\frac{1}{1+2d}\right)
+\frac{1}{2}\pi ^2\Big).
\end{align*}
\end{proposition}
This computation is based on the calculation of the area integrand of the standard quadrilateral $Q_\st$ in Proposition~\ref{prop:intro:integrand:quadrilateral}. We then use the explicit formula from Proposition~\ref{prop: intro hyperbolic} to show that the Holmes--Thompson area achieves a local maximum at $d=1$, roughly equal to $1/4$ of the area $2\pi$ of an ideal quadrilateral in the hyperbolic plane,  and that the area goes to $\frac{3}{8}\pi$ (in particular $>0$) when $d\to 0$ or $d\to \infty$. See Section~\ref{sec:FuchsianLocus} for more details.

\subsubsection{Convex projective structures on a thrice-punctured sphere}
The second family of pairs of inscribed quadrilaterals that we consider arise from the finite area convex projective structures on the thrice-punctured sphere $S_{0,3}$,  with the goal to prove Theorem~\ref{thm: intro: S03}.

Let thus $\rho \in \mathcal{CP}(S_{0,3})$. Consider four points $\xi_1,\ldots,\xi_4 \in \partial\Omega_\rho$ coming from vertices of an adjacent pair of triangles in the lift of the balanced ideal triangulation to $\Omega_\rho$.
An explicit computation, using \cite[Section 4]{BonahonDreyer_ParametrizingHitchinComponents} and the finite area condition, shows that the pair $\mathbf P=(P_{\text{in}},P_{\text{out}})$ of inscribed convex quadrilaterals corresponding to $\xi_1,\ldots,\xi_4$ and the projective lines tangent to $\Omega_\rho$ at these points has Fock--Goncharov parameters $t, t', d,d\in(0,\infty)$ such that $tt'=dd'=1$. See Section \ref{sec:ProjStrSphere} for more details.
We are thus naturally lead to consider pairs $\mathbf P(d,t)=(P_{\text{in}}(d,t),P_{\text{out}}(d,t))$ of inscribed convex quadrilaterals with triple ratios $t,t'$ and double ratios $d,d'$ satisfying $tt'=dd'=1$. 
We study the asymptotic behavior of the area for $d$ or $t$ going to zero or infinity. 
More precisely, in Section~\ref{sec:HTQuadrupleEx2} we prove the following theorem,  which together with the comparison result immediately implies Theorem~\ref{thm: intro: S03}.

\begin{theorem}\label{thm: intro estimate}
Consider the family $\mathbf P(d,t)$ of pairs of inscribed convex quadrilaterals with Fock--Goncharov parameters $t,t',d,d'\in(0,\infty)$ satisfying $tt'=1$ and $dd'=1$. 
Then, there exists $C\in(0,\infty)$ such that
\[
\lim_{\|(\ln(d),\ln(t))\|\to\infty}\frac{\vol^{\HT}(\mathbf P(d,t))}{\ln^2(d)+\ln^2(t)}=C.
\]
\end{theorem}

\subsection{Inscribed triangles}
For convex projective structures on more general surfaces,  not only the thrice-punctured sphere,  we need to compute the Holmes--Thompson area of a pair of inscribed triangles. 
The case of triangles is computationally more simple,  and has already attracted some attention in the context of other notions of area.

The first step is,  as in the case of quadrilaterals,  the computation of the Holmes--Thompson area of subsets of the {\em standard triangle} $T_\st$ with vertices $(0,0)$,  $(0,2)$ and $(2,0)$,  completed in Section~\ref{sec:HTTriple}.

\begin{proposition}
\label{prop:intro:integrand:triangle}
The dual unit ball for the Hilbert metric at a point $(x,y)$ in the interior of 
$T_\st$ has Euclidean area $A_{T_\st}$ equal to
	\[
	A_{T_\st}(x,y) = \frac{3}{2xy(2-x-y)}.
	\]
\end{proposition}

From this we obtain the following direct corollaries.
If $\mathbf{P}$ is a pair of inscribed convex triangles,  then $\mathbf P$ is parametrized by a single invariant $t=t(\mathbf P)>0$,  the triple ratio. We then prove:

\begin{corollary}
\label{prop: intro triangle}
Let $\mathbf P(t)=(P_\inn(t),P_\out(t))$ be a pair of inscribed convex triangles in $\mathbb{R}^2$ with Fock--Goncharov parameter $t\in(0,\infty)$. 
Then, the Holmes--Thompson area of $\mathbf P$ satisfies 
\[
\vol^{\HT}(\mathbf P(t))=\frac{3}{8\pi}\left(\pi^2+\ln^2(t)\right).
\]
\end{corollary}
Again, up to the action of $\mathsf{PGL}_3(\mathbb R)$,  we can normalize $P_\out$ to be equal to $T_\st$. 
We apply the result of Proposition~\ref{prop:intro:integrand:triangle}. The computation of the integral can then be reduced to the integral computed in \cite[Lemma 3.3]{AdeboyeCooper_AreaConvexProjSurfacesFGCoord} since \cite{delaHarpe_OnHilbertsmetricforsimplices} implies that the Holmes--Thomposon area in this case is a scalar multiple of the $p$--area they consider.

As a corollary we state the analogous result of Adeboye--Cooper \cite[Theorem 0.2]{AdeboyeCooper_AreaConvexProjSurfacesFGCoord}---now in terms of the Holmes--Thompson area. 

\begin{corollary}\label{cor:surfaces}
Suppose $\mathcal{T}$ is an ideal triangulation of a punctured surface  $S$ of negative Euler characteristic $\chi(S)$,  and $\rho$ is a convex projective structure on $S$.
Then
\[\vol^\HT(\rho) \geq \frac{3}{8}(-2\pi \chi(S)) + \frac{3}{8\pi} \sum_{i=1}^{-2\chi(S)}\ln^2(t_i),\]
where $t_i$ are the triple ratios of $\rho$ associated to the $-2\chi(S)$ triangles in the ideal triangulation $\mathcal{T}$ of $S$.
\end{corollary}

A similar result holds for closed surfaces of negative Euler characteristic,  when we replace the ideal triangulation by a maximal geodesic lamination \cite{BonahonDreyer_characters}.

In the hyperbolic case (when $t_i=1$ for all $i$), the right-hand side in Corollary~\ref{cor:surfaces} is equal to $\frac{3}{8}(-2\pi\chi(S))$, thus computing exactly $3/8$ of the hyperbolic area.
Note that we have normalized the Holmes--Thompson area in such a way that it agrees with the hyperbolic area in the case that the properly convex domain is an ellipse,  see Section~\ref{sec:HTArea}.

It would be desirable to improve this estimate by taking not only the triple ratios, but also the double ratios into account, as the triple and double ratios (given a choice of topological data on $S$) completely parametrize $\mathcal{CP}(S)$ \cite{FockGoncharov_ModuliSpacesConvexProjectiveStructuresSurfaces, BonahonDreyer_ParametrizingHitchinComponents}.
This idea was our original motivation to consider quadrilaterals. Specifically, it would be interesting to obtain asymptotics or a formula (as in Corollary \ref{prop: intro triangle}) for $\vol^\HT(\mathbf{P})$ when $\mathbf{P}$ is a general pair of inscribed quadrilaterals, although this seems to be quite a demanding task.

\subsection*{Outline of the paper and proof strategies} 
In Section~\ref{sec:Preliminaries} we collect background material on positive tuples of flags,  their relation to pairs of inscribed polygons in $\mathbb{RP}^2$, the Holmes--Thompson volume for Finsler manifolds, and the interplay of these object. 
In particular, we recall a technical result (Theorem~\ref{thm:DualPolygon}) from \cite{LohmanMorrison_PolarsConvexPolygons} which allows us to explicitly find the dual of the unit ball (with respect to the Finsler norm) in the tangent space at a generic point in a convex polygon. 
In Section~\ref{sec:HTTriple} we establish Corollary~\ref{prop: intro triangle} by first computing the integrand for the Holmes--Thompson area integral (Proposition~\ref{prop:intro:integrand:triangle}), and then reducing this computation to an integral computed by Adeboye and Cooper in \cite{AdeboyeCooper_AreaConvexProjSurfacesFGCoord}. 
Section~\ref{sec:HTQuadruple} is dedicated to quadrilaterals.
Again,  using Theorem~\ref{thm:DualPolygon},  we find the general formula for the Holmes--Thompson area for a subset of the standard quadrilateral $Q_\st$ (Proposition~\ref{prop:intro:integrand:quadrilateral}). 
The proof of Proposition~\ref{prop: intro hyperbolic} is carried out by explicit computations in Section~\ref{sec:FuchsianLocus}. 
We break the proof of Theorem~\ref{thm: intro estimate}, which is established in Section~\ref{sec:HTQuadrupleEx2}, into several steps which can be summarized as follows:
(a) find two subsets $T_1,T_2$ of $P_{\text{out}}$ which are contained in $P_{\text{in}}$ and have Holmes--Thompson area comparable to $\ln^{2}(d)$ and $\ln^2(t)$, respectively;
(b) find a non-ideal quadrilateral $Q$ in $P_{\text{out}}$ such that $Q\cap P_{\text{in}}$ contains $T_1,T_2$ and such that the Holmes--Thompson area of $(Q,P_{\text{out}})$ grows like $\ln^{2}(d)+\ln^2(t)$;
(c) Show that $P_{\text{in}} \setminus Q$, which is the union of four triangles in $P_{\text{out}}$, each with one ideal vertex, grows at a rate slower than the minimum of $\ln^{2}(d)$ and $\ln^2(t)$. We finish with the proof of Theorem~\ref{thm: intro: S03} in Section~\ref{sec:ProjStrSphere}.

\subsection*{Acknowledgments}
The authors thank Ilesanmi Adeboye, Caleb Ashley, Samuel Bronstein, and Constantin Vernicos for their interest and helpful discussions. The authors are grateful to the IHES and the IHP for excellent working conditions during their visits in 2023 and 2025 respectively.

The authors acknowledge support of the Institut Henri Poincaré (UAR 839 CNRS-Sorbonne Université), and LabEx CARMIN (ANR-10-LABX-59-01).
This project has received funding from the European Research Council (ERC) under the European Union's Horizon 2020 research and innovation programme (grant agreement No 101018839), ERC GA 101018839. 
The second author acknowledges partial support by an American Mathematical Society (AMS)-
Simons Research Enhancement Grant for PUI Faculty. 

\section{Preliminaries}
\label{sec:Preliminaries}
In this section we introduce the necessary preliminaries on positive tuples of flags, the Hilbert metric on properly convex domains, and their associated Holmes--Thompson volume. We finish the section by recording a result (Theorem \ref{thm:DualPolygon}) on duals of polygons which will be important when computing the Holmes--Thompson volume in the case that the properly convex domain is a polygon.

\subsection{Positive triples and quadruples of flags}
\label{sec:PosTriplesQuadruples} 

Standard references for the material in this subsection
are \cite{FockGoncharov_ModuliSpacesConvexProjectiveStructuresSurfaces, BonahonDreyer_ParametrizingHitchinComponents}.

A {\em (complete) flag} $F$ in $\mathbb R^3$ is a sequence of vector subspaces $\{0\}=F^0\subset F^1\subset F^2\subset F^3=\mathbb R^3$ such that $\dim F^i=i$ for $i=0,1,2,3$. We denote the space of flags by $\mathcal F$.

\begin{notation}\label{not: flag smatrix} It will be convenient to write a flag as a $3$-by-$2$ matrix:
\[
E=\begin{bmatrix}
\vert & \vert\\
e_1&e_2\\
\vert &\vert
\end{bmatrix}
\]
where $e_1$ and $e_2$ are non-zero vectors in $\mathbb R^3$ such that $E^1$ is the span of $e_1$ and $E^2$ is the span of $e_1$ and $e_2$. In this case, we say that the ordered pair of vectors $(e_1,e_2)$ is a representative for the flag $E$.
\end{notation}

A triple of flags $(E,F,G)\in \mathcal F^3$ is in {\em general position} if for any $a,b,c\in \{0,1,2\}$ such that $a+b+c=3$, we have
\[
E^a+ F^b+ G^c=\mathbb R^3.
\]
Similarly, a quadruple of flags $(E,F,G,H)\in\mathcal F^4$ is in {\em general position} if for any $a,b,c,d\in \{0,1,2\}$ such that $a+b+c+d=3$, we have
\[
E^a+F^b+G^c+H^d=\mathbb R^3.
\]
Fix once and for all an identification $\mathsf{\Lambda}^3 \mathbb R^3\cong \mathbb R$. 

\begin{definition}[Triple ratio] Consider a triple of flags $(E,F,G)$ in general position such that
\[
E=\begin{bmatrix}
\vert & \vert\\
e_1&e_2\\
\vert &\vert
\end{bmatrix}\qquad F=\begin{bmatrix}
\vert & \vert\\
f_1&f_2\\
\vert &\vert
\end{bmatrix}\qquad G=\begin{bmatrix}
\vert & \vert\\
g_1&g_2\\
\vert &\vert
\end{bmatrix}.
\]
The {\em triple ratio} of $(E,F,G)$ is 
\begin{equation}\label{eq:triple ratio}
T(E,F,G)\coloneqq\frac{e_1\wedge e_2\wedge f_1}{f_1\wedge g_1\wedge g_2}\cdot \frac{e_1\wedge g_1\wedge g_2}{e_1\wedge f_1\wedge f_2}\cdot \frac{f_1\wedge f_2\wedge g_1}{e_1\wedge e_2\wedge g_1}\in\mathbb R \setminus \{0\}.
\end{equation}
\end{definition}

\begin{definition}[Double ratios] Consider a quadruple of flags $(E,F,G,H)$ in general position such that
\[
E=\begin{bmatrix}
\vert & \vert\\
e_1&e_2\\
\vert &\vert
\end{bmatrix}\qquad F=\begin{bmatrix}
\vert & \vert\\
f_1&f_2\\
\vert &\vert
\end{bmatrix}\qquad G=\begin{bmatrix}
\vert & \vert\\
g_1&g_2\\
\vert &\vert
\end{bmatrix}\qquad H=\begin{bmatrix}
\vert & \vert\\
h_1&h_2\\
\vert &\vert
\end{bmatrix}.
\]
The {\em double ratios} of $(E,F,G,H)$ are 
\begin{align*}
D_1(E,F,G,H)&\coloneqq-\frac{e_1\wedge g_1\wedge f_1}{e_1\wedge g_1\wedge h_1}\cdot \frac{g_1\wedge g_2\wedge h_1}{g_1\wedge g_2\wedge f_1}\in\mathbb R \setminus \{0\},\quad \text{ and}\\
D_2(E,F,G,H)&\coloneqq-\frac{e_1\wedge e_2\wedge f_1}{e_1\wedge e_2\wedge h_1}\cdot \frac{e_1\wedge g_1\wedge h_1}{e_1\wedge g_1\wedge f_1}\in\mathbb R \setminus \{0\}.
\end{align*}
\end{definition}

The triple and double ratios are well-defined and non-zero because the flags are in general position, and their expressions do not depend on the choices of representatives for the flags. Since any two identifications $\mathsf{\Lambda}^3\mathbb R^3\cong \mathbb R$ differ by a non-zero scalar, the triple and double ratios do not depend on this choice. In practice, for $v_1,v_2,v_3\in\mathbb R^3$, one computes $v_1\wedge v_2\wedge v_3$ as the determinant of the matrix with column vectors $v_1$, $v_2$ and $v_3$.

\begin{remark}\label{rmk: combinatorics triple and double} Note that the triple and double ratios are  invariant under the action of the projective general linear group $\mathsf{PGL}_3(\mathbb R)$. 
On the other hand, Fock--Goncharov proved that two triples of flags in general position are in the same $\mathsf{PGL}_3(\mathbb R)$-orbit if and only if they have the same triple ratios. 
Similarly, two quadruples of flags $(E,F,G,H)$ and $(E',F',G',H')$ in general position are in the same $\mathsf{PGL}_3(\mathbb R)$-orbit if and only if they have the same two double ratios, $T(E,F,G)=T(E',F',G')$, and $T(E,G,H)=T(E',G',H')$.

The triple ratio is invariant under cyclic permutation of the flags. The double ratios satisfy the equalities
\begin{gather*}
D_a(E,F,G,H)\cdot D_a(E,H,G,F)=1\text{ for }a=1\text{ and }2,\\ D_{1}(E,F,G,H)\cdot D_{2}(G,F,E,H)=D_{2}(E,F,G,H)\cdot D_{1}(G,F,E,H)=1
\end{gather*}
The asymmetry in the roles of the flags $E, F, G$ and $H$ can be visualized by a choice of an oriented triangulation of a quadrilateral whose vertices are labeled by $E$,  $F$,  $G$ and $H$ (in this counter-clockwise order), and the unique edge of this oriented triangulation goes from $E$ to $G$. If we flip the diagonal from $E$ to $G$ to the diagonal from $F$ to $H$, the change of coordinates is
\begin{gather*}
T(F,E,H)=t' \ \frac{1+d'+d't+d'td}{1+d+dt'+dt'd'},\qquad T(F,H,G)=t \ \frac{1+d+dt'+dt'd'}{1+d'+d't+d'td},\\
D_1(F,E,H,G)=\frac{1+d'}{td'(1+d)}, \qquad D_2(F,E,H,G)=\frac{1+d}{t'd(1+d')},
\end{gather*}
see for example \cite[Section 2.5]{FockGoncharov_ModuliSpacesConvexProjectiveStructuresSurfaces}.

\end{remark}

\begin{definition}[Positivity]\ 
\begin{itemize}
	\item A triple of flags $(E,F,G)$ in general position is {\em positive} if $T(E,F,G)>0$. 

	\item A quadruple of flags $(E,F,G,H)$ in general position is {\em positive} if $(E,F,G)$ and $(E,G,H)$ are positive and $D_a(E,F,G,H)>0$ for $a=1$ and $2$. 
\end{itemize}
\end{definition}

More generally, a $k$-tuple $\mathbf{F}=(F_1,\dots, F_k)$ of flags is in {\em general position} if for all $a_1,\dots,a_k \in \{0,1,2,3\}$ such that $a_1+\dots +a_k=3$, we have
\[
F_1^{a_1}+\dots+F_k^{a_k}=\mathbb R^3.
\]
The $k$-tuple is then {\em positive} if for all $1 \leq i_1<i_2<i_3\leq k$, the triple of flags $(F_{i_1},F_{i_2},F_{i_3})$ is positive,  and for all $1 \leq i_1<i_2<i_3<i_4\leq k$ the quadruple $(F_{i_1},F_{i_2},F_{i_3},F_{i_4})$ is positive.

Fock and Goncharov \cite{FockGoncharov_ModuliSpacesLocalSystemsHigherTeichmuellerTheory} showed, in greater generality, that this notion of positivity for tuples of flags in general position coincides with the one defined by Lusztig via the theory of total positivity in general split Lie groups \cite{Lusztig_TotalPosRedGroups}.

\subsection{Convex polygons from positive tuple of flags}
\label{sec:ConvexPolygonsPositiveTuples}
In \cite{FockGoncharov_ModuliSpacesConvexProjectiveStructuresSurfaces} the notion of positivity is illustrated geometrically as follows. 

Any complete flag $F$ determines uniquely a point $F^1$ in the real projective plane $\mathbb{RP}^2$,  and a projective line $F^2$ in $\mathbb{RP}^2$ passing through $F^1$. 
A triple of flags $(E,F,G)$ is positive if and only if a triangle\footnote{Beware that a triangle in $\mathbb{RP}^2$ (by which we mean a subset that in some affine chart is a bounded triangle) is neither determined by its vertices nor by the projective lines bounding it.} with vertices $E^1$, $F^1$, and $G^1$  is inscribed in a triangle with vertices $E^2\cap F^2$, $F^2\cap G^2$, and $G^2\cap E^2$, see \cite[Lemma 2.3]{FockGoncharov_ModuliSpacesConvexProjectiveStructuresSurfaces}.

For quadruples we have a similar result,  see \cite[Lemma 2.4]{FockGoncharov_ModuliSpacesConvexProjectiveStructuresSurfaces}.
A quadruple $(E,F,G,H)$ is positive if and only if the quadrilateral with vertices $E^1$, $F^1$, $G^1$ and $H^1$ (in this cyclic order) is inscribed in the quadrilateral with vertices $E^2\cap F^2$, $F^2\cap G^2$, $G^2\cap H^2$ and $H^2\cap E^2$ (in this cyclic order). 

More generally for any $k \geq 2$,  a positive $k$-tuple of flags $\mathbf F=(F_1,\ldots,F_k)$ corresponds to a pair of convex $k$-gons,  one inscribed in the other, and such that the vertices occur in the correct cyclic order, see \cite[Theorem 2.2]{FockGoncharov_ModuliSpacesConvexProjectiveStructuresSurfaces} and its proof.
The inner $k$-gon has vertices $F_1^1, \ldots,F_k^1$ (in this cyclic order) and the outer $k$-gon has vertices $F_i^2 \cap F_{i+1}^2$ for $i=1,\ldots,k$ modulo $k$ (in this cyclic order).
With this interpretation we can now define the following.

\begin{definition}
\label{def: polygon from flag}
Let $\mathbf F$ be a positive $k$-tuple of flags in $\mathbb{R}^3$. 
Then we denote by $P_\inn(\mathbf{F}) \subset P_{\out}(\mathbf {F})$ the convex $k$-gons in $\mathbb{RP}^2$ associated to $\mathbf F$.
\end{definition}

Because $\mathbf F$ is in general position, the vertices of $P_\inn(\mathbf{F})$ lie in the interior of the sides of $P_\out(\mathbf{F})$, and every side of $P_\out(\mathbf{F})$ contains exactly one vertex of $P_\inn(\mathbf{F})$. In particular,  $P_\inn(\mathbf{F})$ defines an \emph{ideal} polygon in $P_\out(\mathbf{F})$,  see Figure~\ref{fig:PosTuple}.
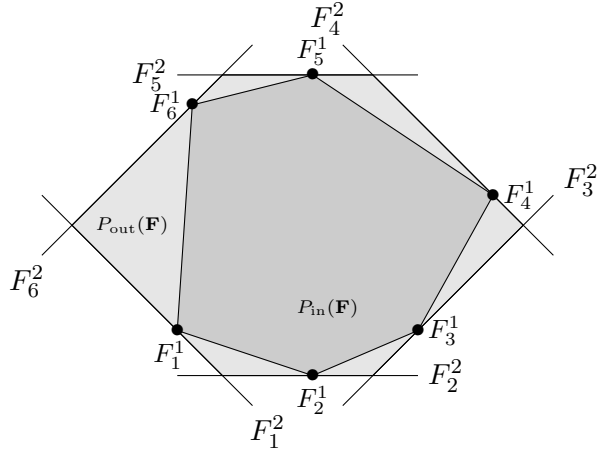
\begin{figure}[h]
\begin{tikzpicture}[scale=2]
	\filldraw[black!10] (0,0) -- (1,0) -- (2,1) -- (1,2) -- (0,2) -- (-1,1) -- cycle;
	\draw (0,0) -- (1,0) -- (2,1) -- (1,2) -- (0,2) -- (-1,1) -- cycle;
	\draw (-.3,0) -- (1.3,0);
	\draw (-.3,2) -- (1.3,2);
	\draw (-1.2,.8) -- (.2,2.2);
	\draw (-1.2,1.2) -- (.2,-.2);
	\draw (.8,-.2) -- (2.2,1.2);
	\draw (2.2,.8) -- (0.8,2.2);
	\draw (-.6,1) node{\tiny $P_{\out}(\mathbf{F})$};
	\filldraw[black!20] (.6,0)  -- (1.3,0.3) -- (1.8,1.2) -- (.6,2) -- (-.2,1.8) -- (-.3,.3) -- cycle;
	\draw (.6,0)  -- (1.3,0.3) -- (1.8,1.2) -- (.6,2) -- (-.2,1.8) -- (-.3,.3) -- cycle;
	\draw (.7,.45) node{\tiny $P_{\inn}(\mathbf{F})$};
	\draw (.3,-.2) node[below]{$F_1^2$};
	\draw (1.3,0) node[right]{$F_2^2$};
	\draw (2.2,1.3) node[right]{$F_3^2$};
	\draw (.7,2.2) node[above]{$F_4^2$};
	\draw (-.3,2) node[left]{$F_5^2 $};
	\draw (-1.3,.8) node[below]{$F_6^2$};
	\draw (.6,0) node{$\bullet$};
	\draw (1.3,0.3) node{$\bullet$};
	\draw (1.8,1.2) node{$\bullet$};
	\draw (.6,2) node{$\bullet$};
	\draw (-.2,1.8) node{$\bullet$};
	\draw (-.3,.3) node{$\bullet$};
	\draw (.6,0) node[below]{\small $F_2^1$};
	\draw (1.3,0.3) node[right]{\small $F_3^1$};
	\draw (1.8,1.2) node[right]{\small $F_4^1$};
	\draw (.6,2) node[above]{\small $F_5^1$};
	\draw (-.2,1.8) node[left]{\small $F_6^1$};
	\draw (-.35,.3) node[below]{\small $F_1^1$};
\end{tikzpicture}
\caption{A positive tuple of flags $\mathbf{F}$ gives rise to two inscribed polygons $P_\inn(\mathbf{F})\subset P_\out(\mathbf{F})$.}
\label{fig:PosTuple}
\end{figure}
\subsection{Hilbert distance}\label{sec: Hilbert prelim} Our estimates for the Holmes--Thompson area of inscribed polygons are largely motivated by the theory of (marked) convex real projective structures on surfaces. 
Such a geometric structure determines a properly convex domain in $\mathbb{RP}^2$. 
That is, an open subset $\Omega$ of the real projective plane whose closure is contained in an affine chart and such that $\Omega$ is convex in this affine chart.
 Given a positive tuple of flags $\mathbf F$, then the interior of $P_{\out}(\mathbf F)$ is an example of a properly convex domain in $\mathbb{RP}^2$.

The Hilbert distance on a properly convex domain $\Omega$ is defined as follows: for every $p\neq q$ contained in $\Omega$, the line connecting these two points intersects the boundary of $\Omega$ in distinct points $a$ and $b$. 
Then, up to relabeling, $a,p,q$, and $b$ appear in this order along a projective line segment in the closure of $\Omega$,  and we set
\[
d_{\Omega}(x,y) \coloneqq \frac{1}{2}\ln [q:p:a:b],
\]
where $[q:p:a:b]$ is the projective cross ratio with the convention $[t:1:0:\infty]=t$.

A properly convex domain $\Omega$ equipped with its Hilbert metric, is an example of a Finsler manifold,  i.e.\ a manifold $M$ endowed with a continuous function $\| \cdot \| \colon TM \to [0,\infty)$ which is smooth away from the zero section,  and in every tangent space restricts to a (not necessarily symmetric) norm.
As for Riemannian manifolds,  this allows to define a (not necessarily symmetric) metric on $M$.
In the case of a properly convex domain $\Omega$, seen as a subset of $\mathbb R^2$, we can define a continuous function
\begin{align*}
\|\cdot\|_\Omega\colon T\Omega \to[0,\infty),  \quad (p,\xi)\mapsto \frac{d}{dt}\Big\vert_{t=0}d_{\Omega}(p,p+t\xi),
\end{align*}
that in every tangent space restricts to a symmetric norm,  and the associated metric is the Hilbert metric.
We can compute the norm $\|\cdot\|_{\Omega}$ explicitly as in \cite[Section 3.2.1]{Benoist_ConvexesDivisiblesI}. 

Specifically, identify $\Omega$ with a subset of $\mathbb R^2$ and consider $v=(p,\xi)\in\Omega\times\mathbb R^2\cong T\Omega$ a tangent vector. 
Let $v^{\pm}$ denote the forward (respectively, backward) intersection points on the boundary of $\Omega$ with the oriented line defined by $v$. Then, there exist positive scalars $\sigma^+(v)$ and $\sigma^-(v)$ such that
\[
\xi=\sigma^+(v)(v^+-p)=\sigma^-(v)(p-v^-),
\]
and the Finsler norm of $v$ equals $\|v\|_\Omega=\frac{1}{2}\left(\sigma^+(v)+\sigma^-(v)\right)$.

\subsection{Holmes--Thompson area}
\label{sec:HTArea}

The Holmes--Thompson volume for Finsler manifolds was first introduced by Holmes--Thompson in \cite{HolmesThompson_NDimAreaContentMinkowskiSpaces}.
It is based on the observation that the cotangent bundle of an $n$-manifold carries a natural symplectic form,  and its $n$-th exterior power is a volume form. For an introduction to volumes on Finsler manifolds we recommend \cite{AlvarezThompson_VolumesNormedFinslerSpaces}.

\begin{definition} The Holmes--Thompson volume of a measurable set $X$ in a Finsler manifold $(M,\|\cdot\|)$ is
\[\vol^\HT(X)\coloneqq\frac{1}{\epsilon_n} \int_{B^*X} \frac 1{n!} \omega ^n,\]
where
$\omega$ is the standard symplectic form on the cotangent bundle $T^*M$ of $M$ and $B^*X=\{(p,\xi)\in T^*M \mid \|\xi\|_p^*\leq 1\}$ is the dual unit ball.
Here $\epsilon_n$ denotes the volume of the $n$-dimensional unit ball.
\end{definition}

The normalization by ${\epsilon_n}$ in the definition of the Holmes--Thompson volume is so that whenever $\|\cdot\|$ comes from a scalar product, in other words, $M$ is Riemannian,  the Holmes--Thompson volume agrees with the Riemannian volume.
This is in general not standardized throughout the literature.

\begin{remark}
Note that for Finsler manifolds,  there is not one canonical volume.
Even though the Holmes--Thompson volume does not agree with the Hausdorff measure of the associated distance (the Busemann volume \cite{Busemann_GeometryFinslerSpaces}), it is natural in many ways, e.g.\ totally geodesic submanifolds are minimal with respect to the Holmes--Thompson area integrand, it satisfies a Crofton formula, etc.
For an interesting exposition of the advantages of the Holmes--Thompson volume compared to other notions of volume we refer to \cite{AlvarezBerck_WhatIsWrongHausdorffMeasureFinslerSpaces}.
\end{remark}

\begin{remark}
\label{rem: volume bilipschitz} Since we ask any reasonable notion of volume on a normed vector space to be translation invariant,  the volume on each tangent space needs to be a multiple of the Lebesgue measure,  see e.g.\  \cite[Section 3]{AlvarezThompson_VolumesNormedFinslerSpaces}.
In particular,  for properly convex domains all notions of volume are bilipschitz equivalent \cite[Proposition 9.4]{Marquis_around},  and asymptotic results do not depend on the notion of volume chosen.
\end{remark}

For our exposition it suffices to compute the Holmes--Thompson volume for subsets of a properly convex domain $\Omega$ of $\R^2$,  that is endowed with its Hilbert metric.
In this case the tangent and cotangent spaces at each point $x\in \Omega$ can both be identified with $\R^2$.
Furthermore,  the scaling factor $\epsilon_2$ equals $\pi$, the area of the unit disc.

\begin{remark}                                                                                                                                                                                                                                                                                                                                                                                                                                                                                                                      The Holmes--Thompson volume of a Borel subset $X\subset \Omega \subset \R^2$ satisfies
\[ \vol_\Omega^\HT(X) = \frac{1}{\pi}|B^*X| = \frac{1}{\pi} \int_X |B^*_p| \,\mathrm d\vol_2(p),\]
where for each point $p\in \Omega$, the set $ B^*_p \subseteq \mathbb R^n$ is the dual unit ball of $(\|\cdot\|_\Omega)_x$, 
$ | \cdot | $ denotes the usual volume of a subset in $\R^2$ and $ \mathrm d \vol_2$ is the standard $2$-dimensional volume.
\end{remark}

\begin{remark}
\label{rem: volume invariant}
 Given a Borel subset $X\subset \Omega\subset \mathbb{RP}^2$ and $g\in\mathsf{PGL}_3(\mathbb R)$, then
\[
\vol_\Omega^\HT(X)=\vol_{g\Omega}^\HT(gX).
\]
\end{remark}

\begin{remark}
\label{rem: volume decreasing}
If $\Omega \subseteq \Omega' \subset \mathbb{RP}^2$ are two properly convex domains,  and $X \subseteq \Omega$,  then
\[ \vol_\Omega^\HT(X) \geq \vol_{\Omega'}^\HT(X).\]
To see this, note that the dual unit ball in $T_p^*\Omega'$ is contained in the dual unit ball $T_p^*\Omega$.
\end{remark}

Since we will not be considering any other notions of volume we write $\vol_\Omega$ instead of $\vol_\Omega^\HT$ in the following.
Furthermore,  we will refer to it as area,  as we stay in the setting of two-dimensional manifolds.

\begin{definition}\label{def: HT area flags}
We define the \emph{Holmes--Thompson area of a positive $k$-tuple of flags $\mathbf F$} in $\R^3$ by
\[  \vol(\mathbf F) \coloneqq \vol_{P_\out(\mathbf F)}(P_\inn(\mathbf F)),\]
where $P_\inn(\mathbf F) \subset P_\out(\mathbf F)$ are the inscribed convex $k$-gons in $\mathbb{RP}^2$ defined by $\mathbf F$ as in Definition~\ref{def: polygon from flag}.
\end{definition}

To be precise,  we compute the volume of the \emph{interior} of $P_\inn(\mathbf F)$ with respect to the Hilbert metric on the \emph{interior} of $P_\out(\mathbf F)$. However, to simplify the exposition we use the same notation for open polygons.

\subsection{Unit tangent balls for the Hilbert distance}
\label{subsection: Unit tangent balls}
We have seen that in order to compute the Holmes--Thompson volume of a subset $X\subseteq \Omega$ for $\Omega \subset \mathbb{RP}^2$ an open properly convex domain endowed with its Hilbert metric,  we need to understand the unit tangent balls and their duals.
There is an important relationship between the shape of $\Omega$ and the unit tangent ball at a point $p \in \Omega$.
For this we need to first introduce a definition due to Papadopoulos--Troyanov.

\begin{definition}[{\cite[Definitions 6.4 and 6.5]{PapadopoulosTroyanov_HarmonicSymmetrizationConvexSets}}]
Let $[a_1,a_2] \subset \mathbb{R}^n$ be a compact segment with $a_1 \neq a_2$,  and $p \in [a_1,a_2]$.
The \emph{harmonic symmetrization of $[a_1,a_2]$ centered at $p$},  written $\mathcal{H}([a_1,a_2],p)$,  is the segment $[b_1,b_2] \subset \mathbb{R}^n$,  where $b_1, b_2 \in \mathbb{R}^n$ satisfy
\begin{itemize}
\item $p$ is the center of $[b_1,b_2]$,
\item $\frac{1}{|p-b_1|}=\frac{1}{2}\left(\frac{1}{|p-a_1|}+\frac{1}{|p-a_2|}\right)$ (with the convention $\frac{1}{0}=\infty$),  and
\item $a_2-a_1 \in \mathbb{R}_{>0} (b_2-b_1)$.
\end{itemize}
The harmonic symmetrization of an open bounded segment $(a_1,a_2) \subset \mathbb{R}^n$ centered at $p$ is then defined as 
\[ \mathcal{H}((a_1,a_2),p) \coloneqq \mathcal{H}([a_1,a_2],p) \cap (a_1,a_2).\]

If $\Omega \subset \mathbb{R}^n$ is a convex and bounded set and $p$ in the interior of $\Omega$,  then the harmonic symmetrization of $\Omega$ centered at $p$ is the set $\mathcal{H}(\Omega,p)$ obtained by replacing each segment of $\Omega$ through $p$ by its harmonic symmetrization centered at $p$.
\end{definition}

The harmonic symmetrization commutes with translation in the sense that $\mathcal{H}(\Omega,p)-p=\mathcal{H}(\Omega-p,o)$,  where $o$ is the origin in $\R^n$. The following theorem (which is in fact not restricted to dimension two) motivates introducing the harmonic symmetrization.

\begin{theorem}[{\cite[Proposition 9.2]{PapadopoulosTroyanov_HarmonicSymmetrizationConvexSets}}]
\label{theorem: Unit Tangent Ball is Harmonic Symmetrization}
Let $\Omega \subset \mathbb{R}^2 \subset \mathbb{RP}^2$ be a properly convex domain endowed with its Hilbert metric.
Then for all $p \in \Omega$ the unit ball $B_p \subset T_p\Omega \cong \mathbb{R}^2$ is equal to the harmonic symmetrization $\mathcal{H}(\Omega-p,o)$ of $\Omega-p$ at the origin $o$. 
\end{theorem}

In special cases,  the harmonic symmetrization can be explicitly computed.
This is for example the case for polygons,  which are of particular interest in our exposition.
More generally,  it is claimed in \cite[Proposition 6.7~(7)]{PapadopoulosTroyanov_HarmonicSymmetrizationConvexSets} that the harmonic symmetrization of a polyhedron is again a polyhedron (in any dimension).
Since we could not find a proof of this fact in the literature we prove a slightly stronger version in the two-dimensional case.

We first have the following observation,  relating the harmonic symmetrization of a compact segment in $\mathbb{R}$ about $0$ to harmonic cross ratios.

\begin{lemma}
\label{lemma: Harmonic symmetrization and cross-ratio}
Let $[a_1,a_2] \subset \mathbb{R} \subset \mathbb{RP}^1$ be a compact subset containing $0$.
Then the harmonic symmetrization of $[a_1,a_2]$ with respect to $0$ is equal to $[-b,b]$, where $b>0$ is such that $[0:b:a_2:-a_1]=-1$,  with the following definition of the cross ratio $[a:b:c:d]=\frac{c-a}{c-b}\cdot \frac{d-b}{d-a}$.
\end{lemma}
\begin{proof}
The first and third condition in the definition of the harmonic symmetrization are trivially satisfied.
For the second condition we compute
\[[0:b:a_2:-a_1]=\frac{a_2}{a_2-b}\cdot\frac{-a_1-b}{-a_1}=\frac{1}{1-\frac{b}{a_2}}\cdot \left(1+\frac{b}{a_1}\right).\]
Setting this cross ratio equal to $-1$,  and recalling that $a_1<0$ and $b,a_2>0$,  gives
\[2=\frac{b}{a_2}-\frac{b}{a_1}=|b|\left(\frac{1}{|a_1|}+\frac{1}{|a_2|}\right),\]
which is equivalent to the second condition in the definition of the harmonic symmetrization.
\end{proof}

\begin{lemma}
\label{lemma: Harmonic Symmetrization of Polygon}
Let $\ell$ and $\ell'$ be two distinct lines through the origin $o$ in $\R^2$. 
Let $L_1$ a line that intersects $\ell$ and $\ell'$ in the lower half-plane,  and $s_1$ the segment on $L_1$ with endpoints the intersection points.
Similarly, let $L_2$ be a line that intersects $\ell$ and $\ell'$ in the upper half-plane,  and $s_2$ the segment on $L_2$ with endpoints the intersection points.

Then, there exists segments $r_1$ and $r_2$ with endpoints on $\ell$ and $\ell'$,  such that the harmonic symmetrization centered at $o$ of every segment passing through $o$ with an endpoint in $s_1$ and the other in $s_2$ has an endpoint in $r_1$ and the other in $r_2$.
\end{lemma}
\begin{proof}
Consider the segments $-s_1$ and $s_2$ in $\mathbb{R}^2$.
Then $-L_1$ is the line containing $-s_1$.
The lines $-L_1$ and $L_2$ (now seen as projective lines in $\mathbb{RP}^2$) intersect in a unique point $q \in \mathbb{RP}^2$,  where we identify $\mathbb{R}^2$ with a fixed affine chart in $\mathbb{RP}^2$.
Consider the projective line $L_0$ passing through $q$ and $o$.
The harmonic condition 
\[[L_0:L:L_2:-L_1]=-1\]
uniquely defines a projective line $L$ passing through $q$ (as the cross ratio is an invariant of four projective lines passing through a common point).
Now $\ell$ (viewed as a projective line in $\mathbb{RP}^2$) intersects $L$ in a unique point $a$.
Similarly,  $\ell'$ intersects $L$ in a unique point $a'$.
Let $r_2 \coloneqq [a,a'] \subset L$,  and $r_1 \coloneqq -r_2$.
We claim that these are the desired segments.

Let now $s$ be another segment passing through $o$ with endpoints in $s_1$ and $s_2$.
Then the projective line $\ell_s$ spanned by $s$ intersects $L_0$,  $L_1$,  $L_2$ and $L$.
By the property of the cross ratio,  the respective intersection points $o,x,x_2,x_1$ satisfy $[o:x:x_2:x_1]=-1$.
Thus by Lemma~\ref{lemma: Harmonic symmetrization and cross-ratio},  the point $x=\ell_s \cap L$ is one of the endpoints of the harmonic symmetrization of $s$ (and $-x$ is the other one by symmetry).
Now $x \in r_2$ (and then also $-x \in -r_2=r_1$),  because $s$ has endpoints in $s_1$ and $s_2$,  so $\ell_s$  lies in between $\ell$ and $\ell'$.
\end{proof}

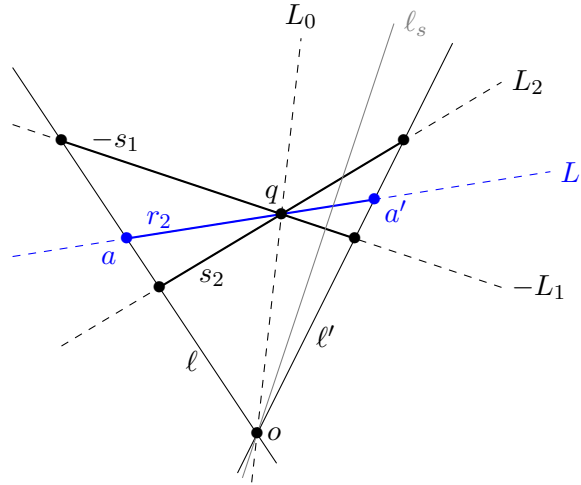
\begin{figure}[h]
\begin{tikzpicture}[scale=1.3]
	\draw (-.2,-.4) -- (2,4); 
	\draw[right] (.5,1) node{$\ell'$};
	\draw[thick] (1.5,3) -- (-1,1.5); 
	\draw[right] (-.7,1.6) node{$s_2$};
	\draw[dashed] (-2,9/10) -- (5/2,36/10); 
	\draw[right] (5/2,36/10) node{$L_2$};
	\draw (-2.5,3.75) --(0.2,-0.3); 
	\draw[left] (-.5,.75) node{$\ell$};
	\draw[thick] (-2,3) -- (1,2); 
	\draw[dashed] (-5/2,19/6) -- (5/2,3/2); 
	\draw[right] (5/2,3/2) node{$-L_1$};
	\draw[right] (-1.8,3) node{$-s_1$};
	\draw[dashed] (-.055,-.055*9)--(2.7/6,2.7/6*9); 
	\draw[above] (2.7/6,2.7/6*9) node{$L_0$};
	\draw[blue, thick] (-4/3,2) --  (6/5,12/5); 
	\draw[blue, above] (-1,2) node{$r_2$};
	\draw[blue, dashed] (-5/2,69/38) --  (3,51/19); 
	\draw[right, blue] (3,51/19) node{$L$};
	\draw (.25,9/4) node{$\bullet$}; 
	\draw[above] (.15,9/4) node{$q$}; 
	\draw (1.5,3) node{$\bullet$};
	\draw (-1,1.5) node{$\bullet$};
	\draw (1,2) node{$\bullet$};
	\draw (-2,3) node{$\bullet$};
	\draw[blue] (-4/3,2) node{$\bullet$};
	\draw[blue,left] (-4/3,1.8) node{$a$};
	\draw[blue] (6/5,12/5) node{$\bullet$};
	\draw[blue,right] (6/5,11.4/5) node{$a'$};
	
	\draw[gray] (-.15,-.45)--(1.4,4.2);
	\draw[gray,right] (1.4,4.2) node{$\ell_s$};
	
	\draw (0,0) node{$\bullet$};
	\draw[right] (0,0) node{$o$};
\end{tikzpicture}
\caption{Visualizing the harmonic symmetrization and the proof of Lemma~\ref{lemma: Harmonic Symmetrization of Polygon}.}
\end{figure}

We can now prove the desired proposition.

\begin{proposition}
\label{proposition: Harmonic Symmetrization of Polygon}
Let $P$ be a $k$-polygon in $\mathbb{R}^2$ containing the origin $o$ in its interior.
Then the harmonic symmetrization of $P$ centered at $o$ is a $k'$-polygon with $k \leq k' \leq 2k$.

Furthermore,  the vertices of $\mathcal{H}(P,o)$ lie on the lines passing through $o$ and a vertex of $P$.
\end{proposition}
\begin{proof}
Let us denote by $z_1,\ldots,z_k$ the vertices of $P$.
For each vertex $z_j$ consider the line passing through $z_j$ and $o$. (You can see the construction in Figure~\ref{fig:UnitBallQst} for the polygon $Q_\st-p$.)
Denote the other intersection point of this line with $\partial P$ by $z_j'$.
Now label all points $z_j$ and $z_j'$ by $y_1,\ldots,y_{k'}$ going in this cyclic order around $\partial P$.
Note that $k \leq k' \leq 2k$.

We claim that for all $j=1,\ldots,k'$ the segment $s_j$ between $y_j$ and $y_{j+1}$ (this is contained in the boundary $\partial P$) is mapped under the harmonic symmetrization centered at $o$ to a segment.
Since the union of the $s_j$ for all $j=1,\ldots,m$ (with the convention that $y_{m+1}=y_1$) covers $\partial P$,  we obtain that $\mathcal{H}(P,o)$ is a polygon.

The proof of this claim follows easily from Lemma~\ref{lemma: Harmonic Symmetrization of Polygon}.
Indeed, fix $j$ with $1\leq j \leq k'$.
Then $j$ determines four points among the $y_i$,  namely $y_j$, $y_{j+1}$,  and the intersection points of the rays starting from $y_j$ respectively $y_{j+1}$ passing through $o$.
These correspond to $y_m$ and $y_{m+1}$ respectively for some $m \neq j-1,j,j+1$.
By construction,  the respective segments $s_j$ and $s_m$ satisfy, up to rotation, the conditions of Lemma~\ref{lemma: Harmonic Symmetrization of Polygon}, which proves the claim.
Thus $\mathcal{H}(P,o)$ is a polygon.

The segments $s_j$ have, by definition, both of their endpoints on lines passing through $o$ and a vertex of $P$ (even more, at least one of its endpoints is in the vertex set of $P$).
By the construction in the proof of Lemma~\ref{lemma: Harmonic Symmetrization of Polygon}, the endpoints of $r_j$,  the image of $s_j$ under the harmonic symmetrization centered at $o$,  lie on the same lines.
This proves the second statement.
\end{proof}

\subsection{Duals of polygons}
\label{sec:DualPolyg}
In the following sections we need to compute Holmes--Thompson areas in the case when the properly convex domain $\Omega$ is a polygon.
In this case,  the unit tangent ball for the Hilbert distance at every point is a symmetric polygon, as we have seen in the last subsection. Therefore we need to be able to compute duals of polygons, which are again polygons.
There is an explicit formula relating the vertices of a convex symmetric polygon and its dual polygon.
We follow the algorithm described in \cite{LohmanMorrison_PolarsConvexPolygons}.

Let thus $P \subset \R^2$ be a symmetric polygon centered at $o=(0,0)$ with $2k$ vertices.
Label the vertices  of $P$ that lie above the horizontal axis or on the non-negative horizontal axis from left to right by $z_1, \ldots, z_k$.
We also set $z_0 \coloneqq -z_k$ and write
$z_i=(x_i,y_i)$ for all $i=0,\ldots,k$.

\begin{theorem}[\cite{LohmanMorrison_PolarsConvexPolygons}]
\label{thm:DualPolygon}
The dual polygon of the polygon with vertices $\pm z_i$ has vertices $\pm w_i$ for $i=1,\ldots,k$,  where
\[w_i = \frac{1}{x_{i-1}y_i-x_iy_{i-1}}\begin{pmatrix}
y_i-y_{i-1}\\
x_{i-1}-x_i
\end{pmatrix}.\]
\end{theorem}

\section{Holmes--Thompson area of a positive triple of flags}
\label{sec:HTTriple}
The goal of this section is the proof of the following proposition.

\begin{proposition}[Corollary~\ref{prop: intro triangle}]
\label{prop: VolumePositiveTriple}
Let $\mathbf F$ be a positive triple of flags in $\mathbb R^3$ with triple ratio $t \in (0,\infty)$.
Then
\[\vol(\mathbf{F})=\frac{3}{8\pi}(\pi^2+\ln^2(t)).\]
\end{proposition}

\begin{remark}
The result that $\vol({\bf F})$ grows like $\ln^2(t)$ follows already from the works of \cite[Proposition 0.3]{AdeboyeCooper_AreaConvexProjSurfacesFGCoord} (Adeboye and Cooper obtain $\frac{1}{8}(\pi^2 +\ln^2(t))$ for the $p$-area) or \cite[Proposition 3]{Wolpert_HilbertAreaInscribedTrianglesQuadrilaterals} (Wolpert shows that there exist positive constants $c_1$ and $c_2$ such that the Hausdorff measure of $T$ is bounded above and below by $c_1 (1+\ln^2(t))$ respectively $c_2 (1+\ln^2(t))$). Even though they compute a different notion of area, de la Harpe \cite{delaHarpe_OnHilbertsmetricforsimplices} showed that a simplex with its Hilbert metric has the structure of a normed vector space (see Remark~\ref{rem: volume bilipschitz}). Thus, all notions of area in question are scalar multiples of each other and we get the same asymptotic behavior.
\end{remark}

Before we prove the above proposition,  we do some preliminary considerations and normalizations.
We have seen in Section~\ref{sec:ConvexPolygonsPositiveTuples} how a positive triple of flags defines two triangles, one inscribed into the other. Up to the action of $\mathsf{PGL}_3(\bb R)$ we can assume that the outer triangle is the standard triangle $T_\st$ with vertices
\[
v_1=\begin{pmatrix}
0\\0
\end{pmatrix}\qquad
v_2=\begin{pmatrix}
2\\0
\end{pmatrix}\qquad
v_3=\begin{pmatrix}
0\\2
\end{pmatrix},
\]
and that the inner triangle $T$ has vertices 
\[
w_1=\begin{pmatrix}
1\\1
\end{pmatrix}\qquad
w_2=\begin{pmatrix}
1\\0
\end{pmatrix}\qquad
w_3=\begin{pmatrix}
0\\s
\end{pmatrix},
\]
where $s= \frac{2t}{t+1} \in (0,2)$ and $t$ is the triple ratio of $\mathbf F$.
Indeed, consider the flags (see Notation \ref{not: flag smatrix})
\begin{gather*}
E=\begin{bmatrix}1&2\\1&0\\1&1\end{bmatrix}\qquad F=\begin{bmatrix}1&0\\0&0\\1&1\end{bmatrix}\qquad
G=\begin{bmatrix}0&0\\s&0\\1&1\end{bmatrix}
\end{gather*}
in the $\PGL_3(\mathbb{R})$-orbit of $\mathbf F$. 
Using the formula in (\ref{eq:triple ratio}),
we find
\[
t\coloneqq T({\bf F})=T(E,F,G)=\frac{s}{2-s},  \textnormal{ or,  equivalently } s = \frac{2t}{t+1}.
\]

\begin{figure}[h]
\begin{tikzpicture}[scale=2]
	\draw[->] (-.5,0) -- (2.5,0);
	\draw[->] (0,-.5) -- (0,2.5);
	\filldraw[black!10] (0,0) -- (2,0) -- (0,2) -- cycle;
	\draw (0,0) -- (2,0) -- (0,2) -- cycle;
	\draw (.4,1.1) node{$T_\st$};
	\filldraw[black!20] (1,1) -- (1,0) -- (0,.3) -- cycle;
	\draw (1,1) -- (1,0) -- (0,.3) -- cycle;
	\draw (.7,.45) node{$T$};
	\draw (0,0) node{$\bullet$};
	\draw (2,0) node{$\bullet$};
	\draw (0,2) node{$\bullet$};
	\draw (1,1) node{$\bullet$};
	\draw (1,0) node{$\bullet$};
	\draw (0,0.3) node{$\bullet$};
	\draw[left] (0,-.25) node{\small{$\begin{pmatrix}0\\0\end{pmatrix}$}};
	\draw[below] (2,0) node{\small{$\begin{pmatrix}2\\0\end{pmatrix}$}};
	\draw[left] (0,2) node{\small{$\begin{pmatrix}0\\2\end{pmatrix}$}};
	\draw[below] (1,0) node{\small{$\begin{pmatrix}1\\0\end{pmatrix}$}};
	\draw[right] (1,1.1) node{\small{$\begin{pmatrix}1\\1\end{pmatrix}$}};
	\draw[left] (0,0.3) node{\small{$\begin{pmatrix}0\\s\end{pmatrix}$}};
\end{tikzpicture}
\caption{A positive triple of flags gives rise to two inscribed triangles $T \subset T_\st$.}
\end{figure}
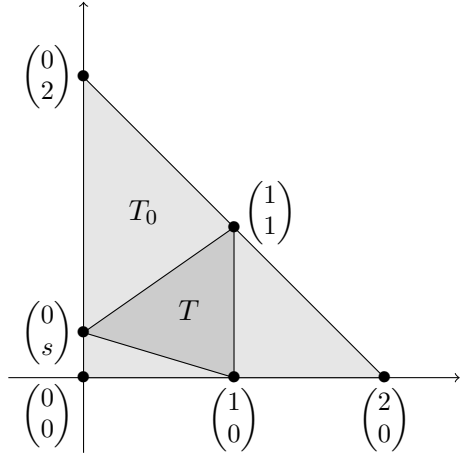

\subsection{Unit ball for the Hilbert metric of a triangle}
Recall from Section~\ref{subsection: Unit tangent balls} that for any $p=(x,y)\in T_\st$ the unit ball for the Hilbert metric around $p$ is a polygon with at most six vertices.
Moreover,  it follows from the proof of Theorem~\ref{theorem: Unit Tangent Ball is Harmonic Symmetrization},  Proposition~\ref{proposition: Harmonic Symmetrization of Polygon} and the translation invariance of the harmonic symmetrization,  that we can explicitly compute the vertices of the unit ball in the tangent space $T_p T_\st$ as follows. Let $z$ be one of the vertices of $T_\st$. 
Then, let $v_z=(p,\xi_z)$ be the unit tangent vector at $p$ in the direction of $z$. 
Denote the endpoints of the corresponding geodesic line by $v^{\pm}_z$ and note that $v_z^{+}=z$ by convention. 
Then, since $v_z$ has unit norm, there exists $\sigma_z\in\bb R$ such that 
\[
\xi_z=\sigma_z(v_z^+-p)=(2-\sigma_z)(p-v_z^-).
\]
With these notations, we can explicitly compute the vectors $\xi_z\in\bb R^2$ for $z$ a vertex of $T_\st$. 

For $z_1 = (0,0)$ we obtain the equality
\[
\sigma_{z_1}\left(-\begin{pmatrix}x\\y\end{pmatrix}\right)=(2-\sigma_{z_1})\left(\begin{pmatrix}x\\y \end{pmatrix}-\frac{2}{x+y}\begin{pmatrix}x\\y\end{pmatrix}\right),
\] 
from which we deduce
\[
\xi_{(0,0)}=(x+y-2)\begin{pmatrix}x\\y\end{pmatrix}.
\]
Similar computations give 
\[
\xi_{(2,0)}=\begin{pmatrix}x(2-x)\\-xy\end{pmatrix} \qquad\text{and}\qquad \xi_{(0,2)}=\begin{pmatrix}-xy\\y(2-y)\end{pmatrix}.
\]

\subsection{Dual unit ball for the Hilbert metric of a triangle}
\label{subsec:DualUnitBallTriangle}

We can now prove the following lemma.
\begin{lemma}[Proposition~\ref{prop:intro:integrand:triangle}]
\label{lemma: integrand triangle}
The area of the dual unit ball at a point $(x,y) \in T_\st$ is equal to
\begin{equation}\label{eq: area triangle}
A_{T_0}(x,y) \coloneqq \frac{3}{2xy(2-x-y)}.
\end{equation}
\end{lemma}
\begin{proof}
We have seen in the last section how to compute the vertices of the unit ball.
We would like to apply Theorem \ref{thm:DualPolygon} to the hexagon with vertices $\pm \xi_{(0,0)},\pm \xi_{(2,0)},  \pm \xi_{(0,2)}$.
The vertices $-\xi_{(0,0)}$, $\xi_{(0,2)}$,  and $-\xi_{(2,0)}$ lie above the horizontal axis.
Denote by $\prec$ the order on the first coordinates of vectors in $\bb R^2$.
Then 
\[
-\xi_{(2,0)}\prec \xi_{(0,2)} \prec -\xi_{(0,0)}.
\] 
Thus applying Theorem \ref{thm:DualPolygon}, we obtain that the dual polygon has vertices
\[ \pm w_1 = \pm\frac{1}{2(x+y-2)} \begin{pmatrix}
\tfrac{2-y}{x} \\ 1
\end{pmatrix}, \ \pm w_2 = \pm \begin{pmatrix}
-\tfrac{1}{2x}\\ \tfrac{1}{2y}
\end{pmatrix}, \ \pm w_3 = \pm  \frac{1}{2(x+y-2)}\begin{pmatrix}
-1\\ \tfrac{x-2}{y}
\end{pmatrix}.
\]
Note that all denominators are non-zero as $p=(x,y)$ lies in the interior of $T_\st$; in particular we have $x,y >0$ and $x + y <2$.

To compute the Euclidean area of the dual hexagon we start by noticing that 
\[
-w_1=\frac{1}{2(2-x-y)}\begin{pmatrix}\tfrac{2-y}{x}\\1\end{pmatrix}, \qquad w_3=\frac{1}{2(2-x-y)}\begin{pmatrix}1\\\tfrac{2-x}{y}\end{pmatrix},\qquad w_2=\begin{pmatrix}-\tfrac{1}{2x}\\\tfrac{1}{2y}.
\end{pmatrix}
\] 
are all in the upper half plane since $x,y>0$ and $x+y<2$. Moreover, $w_2\prec w_3\prec -w_1$. Using the fact that the hexagon with vertices $\pm w_1,\pm w_2,$ and $\pm w_3$ is symmetric about the origin, we apply the shoelace formula to obtain
\begin{align*}
A_{T_0}(x,y)&=\det(-w_1, w_3)+\det(w_3, w_2)+\det(w_2, w_1),
\end{align*}
where $\det(v,w)$ denotes the determinant of the 2-by-2 matrix with columns the vectors $v$ and $w$. 
Thus the area of the dual unit ball is equal to $\frac{3}{2xy(2-x-y)}$ as desired.
\end{proof}

\subsection{Holmes--Thompson area of $T$} Note that up until now we have not required that the point $p$ lies in the inscribed triangle $T$.
We will have to do so when we compute 
\[
\vol(\bf{F}) = \vol_{T_\st}^\HT(T) = \frac{1}{\pi}\int_{T}A_{T_0}(x,y)d(x,y).
\]
Since $T_0$ equipped with its Hilbert metric is a normed vector space by \cite{delaHarpe_OnHilbertsmetricforsimplices}, the Holmes--Thompson area of $T$ differs from its $p$-area by a positive scalar independent on the parameter $t$ (and, equivalently, $s$). Thus, applying \cite{AdeboyeCooper_AreaConvexProjSurfacesFGCoord}, there exists $\mu>0$ such that for every $t\in(0,\infty)$
\[
\vol(\bf{F})=\frac{\mu}{8}\left(\pi^2+\ln^2(t)\right).
\]
We compute $\mu$ in the case $t=s=1$. Then, $p=(x,y) \in T$ if and only if
\begin{gather*}
0<x<1,  \quad -x + 1 < y <1.
\end{gather*}
We are now ready to prove Proposition~\ref{prop: VolumePositiveTriple}.
\begin{proof}[Proof of Proposition~\ref{prop: VolumePositiveTriple}]
We compute explicitly
\begin{align*}
\frac{2}{3}\int_{T}A_{T_0}(x,y)d(x,y)
&= \int_0^1 \int_{-x+1}^{1} \frac{1}{xy(2-x-y)} dy dx= \int_0^1 \int_{-x+1}^{1} \frac{1}{x(2-x)} \left(\frac{1}{y}+\frac{1}{2-x-y} \right) dy dx\\
&=-2\int_0^1  \frac{\ln(1-x)}{x(2-x)} dx=-2\int_0^1 \frac{\ln(u)}{1-u^2}du=-\int_0^1\frac{\ln(u)}{1-u}du-\int_0^1\frac{\ln(u)}{1+u}du=\frac{\pi^2}{4}.
\end{align*}
This computation shows that the Holmes--Thompson area of ${\bf F}$ when $t=1$ equals $\tfrac{3\pi}{8}$ while its $p$-area is $\tfrac{\pi^2}{8}$.
Thus, $\mu=\tfrac{3}{\pi}$ and  $\vol({\bf F}) = \frac{3}{8\pi} \left( \pi^2 + \ln^2(t) \right)$.
\end{proof}

\section{Holmes--Thompson area of a positive quadruple of flags}
\label{sec:HTQuadruple}

In this section we establish a formula to compute the Holmes--Thompson area of a positive quadruple of flags $\mathbf F$ in $\mathbb R^3$ (see Definition \ref{def: HT area flags}) in terms of its positive triple and double ratios. Up to the action of $\PGL_3(\mathbb R)$, we can normalize so that the outer quadrilateral of $\mathbf F$ is the standard quadrilateral $Q_0\subset \mathbb R^2$ with vertices 
\[
z_1=\begin{pmatrix}
1\\1
\end{pmatrix}\qquad
z_2=\begin{pmatrix}
-1\\1
\end{pmatrix}\qquad
z_3=\begin{pmatrix}
-1\\-1
\end{pmatrix}\qquad
z_4=\begin{pmatrix}
1\\-1
\end{pmatrix}.
\] 
The following Proposition, proven in Sections \ref{sec: unit ball quad} and \ref{sec: dual unit ball quad}, is the key technical contribution of this section.

\begin{proposition}[Proposition~\ref{prop:intro:integrand:quadrilateral}]
\label{prop: quad area}
The area of the dual unit ball at a point $(x,y)\in Q_0\setminus\{y=\pm x\}$ is equal to 
\[
A_{Q_\st}(x,y)\coloneqq\frac{2+\max\{|x|,|y|\}}{2(1-x^2)(1-y^2)}.
\]
\end{proposition}
As illustrated in Figure~\ref{fig:PosQuadruple}, after normalizing the outer quadrilateral of $\mathbf F$ so that it equals $Q_0$, there exist $\alpha_1$, $\alpha_2$, $\beta_1$, and $\beta_2$ in $(-1,1)$ such that the inner quadrilateral $Q$ has vertices
\[
w_1=\begin{pmatrix}
\beta_1\\1
\end{pmatrix}\qquad w_2=\begin{pmatrix}
-1\\\alpha_1
\end{pmatrix}\qquad
w_3=\begin{pmatrix}
\beta_2\\-1
\end{pmatrix}\qquad
w_4=\begin{pmatrix}
1\\\alpha_2
\end{pmatrix}.
\]
Proposition~\ref{prop: quad area} gives a way to compute $\vol({\bf F})$ in terms of the Fock--Goncharov parameters of ${\bf F}$.

\tikzmath{\x = .5; \y =.1; \z = .2;}
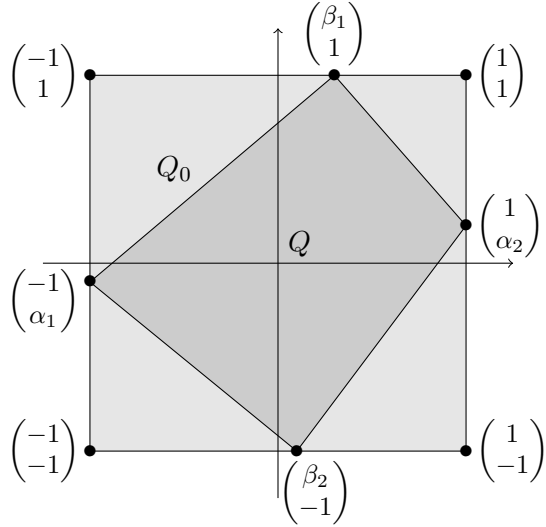
\begin{figure}[h]
\begin{tikzpicture}[scale=2.5]
\filldraw[fill=black!10] (-1,-1) -- (1,-1) -- (1,1) -- (-1,1)-- cycle;
\filldraw[fill=black!20] (1,.2) -- (.3,1) -- (-1,-.1) -- (.1, -1) --cycle;
\draw[->] (-1.25,0) -- (1.25,0);
\draw[->] (0,-1.25) -- (0,1.25);
\draw[right] (0,.1) node{$Q$};
\draw[right] (-.7,.5) node{$Q_\st$};
\draw (1,1) node{$\bullet$};
\draw (1,1) node[right]{\small $\begin{pmatrix}1\\1\end{pmatrix}$};
\draw (1,-1) node{$\bullet$};
\draw (1,-1) node[right]{\small $\begin{pmatrix}1\\-1\end{pmatrix}$};
\draw (-1,-1) node{$\bullet$};
\draw (-1,-1) node[left]{\small $\begin{pmatrix}-1\\-1\end{pmatrix}$};
\draw (-1,1) node{$\bullet$};
\draw (-1,1) node[left]{\small $\begin{pmatrix}-1\\1\end{pmatrix}$};
\draw (1,.2) node{$\bullet$};
\draw (1,.2) node[right]{\small $\begin{pmatrix} 1\\ \alpha_2 \end{pmatrix}$};
\draw (.3,1) node{$\bullet$};
\draw (.3,1) node[above]{\small $\begin{pmatrix} \beta_1\\1 \end{pmatrix}$};
\draw (-1,-.1) node{$\bullet$};
\draw (-1,-.2) node[left]{\small $\begin{pmatrix} -1\\ \alpha_1 \end{pmatrix}$};
\draw (.1,-1) node{$\bullet$};
\draw (.2,-1) node[below]{\small $\begin{pmatrix}\beta_2\\-1\end{pmatrix}$};
\end{tikzpicture}
	\caption{A positive quadruple of flags gives rise to two inscribed quadrilaterals $Q \subset Q_\st$.}
		\label{fig:PosQuadruple}
\end{figure}

\begin{corollary}
\label{prop: VolumePositiveQuadruple}
Let ${\bf F}=(E,F,G,H)$ be a positive quadruple of flags in $\mathbb R^3$ with triple ratios $t=T(E,F,G)$,  $t'=T(E,G,H)$ and double ratios $d=D_1(E,F,G,H)$,  $d'=D_2(E,F,G,H)$.
Then 
\[
\vol({\bf F})=\frac{1}{\pi}\int_{Q}A_{Q_0}(x,y)dxdy,
\]
where $Q$ is the convex quadrilateral inscribed in the standard quadrilateral $Q_0$ with vertices 
\[
\begin{pmatrix}
\beta_1\\1
\end{pmatrix},\quad
\begin{pmatrix}
-1\\\alpha_1
\end{pmatrix},\quad
\begin{pmatrix}
\beta_2\\-1
\end{pmatrix},\quad
\begin{pmatrix}
1\\\alpha_2
\end{pmatrix}.
\]
satisfying
\begin{gather}\label{eq: ai bi via triple and double}
\begin{split}
\alpha_1= \frac{d(1-t)\cdot \frac{1+t'}{1+t}+t'-dd'}{d(1+t)\cdot \frac{1+t'}{1+t}+t'+dd'}
\qquad 
\alpha_2=\frac{d'(1-t')\cdot \frac{1+t}{1+t'}+t-dd'}{d'(1+t')\cdot \frac{1+t}{1+t'}+t+dd'}, \\
\beta_1=-\frac{t'-d\cdot \frac{1+t'}{1+t}}{t'+d\cdot \frac{1+t'}{1+t}},\qquad
\beta_2=\frac{t-d'\cdot \frac{1+t}{1+t'}}{t+d'\cdot \frac{1+t}{1+t'}}.\qquad\qquad\quad
\end{split}
\end{gather}
\end{corollary}
\begin{proof}
The first assertion follows from the definition of Holmes--Thompson area and Proposition~\ref{prop: quad area}. 
We complete the proof of the corollary by relating the parameters $\alpha_1$, $\alpha_2$, $\beta_1$, and $\beta_2$ to the Fock--Goncharov invariants of ${\bf F}$. 
Explicitly, we are using an element of $\mathsf{PGL}_3(\mathbb R)$ to send the quadruple ${\bf F}$ to the quadruple $(E',F',G',H')$ defined by (see Notation~\ref{not: flag smatrix})
\begin{gather*}
E'=\begin{bmatrix}\beta_1&1\\1&1\\1&1\end{bmatrix}=\begin{bmatrix}\beta_1&1-\beta_1\\1&0\\1&0\end{bmatrix}\qquad F'=\begin{bmatrix}-1&-1\\\alpha_1&-1\\1&1\end{bmatrix}=\begin{bmatrix}-1&0\\\alpha_1&1+\alpha_1\\1&0\end{bmatrix}\\
G'=\begin{bmatrix}\beta_2&-1\\-1&-1\\1&1\end{bmatrix}=\begin{bmatrix}\beta_2&1+\beta_2\\-1&0\\1&0\end{bmatrix}  \qquad H'=\begin{bmatrix}1&1\\\alpha_2&1\\1&1\end{bmatrix}=\begin{bmatrix}1&0\\\alpha_2&1-\alpha_2\\1&0\end{bmatrix}.
\end{gather*}
Then, using the definitions of the triple and double ratios as in Section~\ref{sec:PosTriplesQuadruples}, we find
\begin{gather*}
t=T(E',F',G')=\frac{1-\alpha_1}{1+\alpha_1}\cdot\frac{1+\beta_2}{1+\beta_1},  \qquad t' = T(E', G', H')=\frac{1+\alpha_2}{1-\alpha_2}\cdot\frac{1-\beta_1}{1-\beta_2}, \\
d = D_1(E',F',G',H')=\frac{1+\alpha_2}{1+\alpha_1}\cdot\frac{(1+\alpha_1)(1+\beta_1)+(1-\alpha_1)(1+\beta_2)}{(1+\alpha_2)(1-\beta_1)+(1-\alpha_2)(1-\beta_2)}=\frac{1+\beta_1}{1-\beta_1}\cdot t'\cdot \frac{1+t}{1+t'},\\
   d'= D_2(E',F',G',H')=\frac{1-\alpha_1}{1-\alpha_2}\cdot\frac{(1+\alpha_2)(1-\beta_1)+(1-\alpha_2)(1-\beta_2)}{(1+\alpha_1)(1+\beta_1)+(1-\alpha_1)(1+\beta_2)}=\frac{1-\beta_2}{1+\beta_2}\cdot t\cdot \frac{1+t'}{1+t}.
\end{gather*}
The identities~(\ref{eq: ai bi via triple and double}) follow by solving for $\alpha_1$, $\alpha_2$, $\beta_1$, $\beta_2$ in terms of these double and triple ratios.
\end{proof}

\subsection{Unit ball for the Hilbert metric}\label{sec: unit ball quad}
We start our proof of Proposition~\ref{prop: quad area} by finding a formula for the vertices of the unit ball at a point $p=(x,y)$ in the $Q_0\setminus\{y=\pm x\}$. As in Section~\ref{subsec:DualUnitBallTriangle}, we explicitly compute the vertices of the unit ball in $T_pQ_\st$ for $p=(x,y)$ in the standard quadrilateral $Q_\st \setminus \{y=\pm x\} $.  
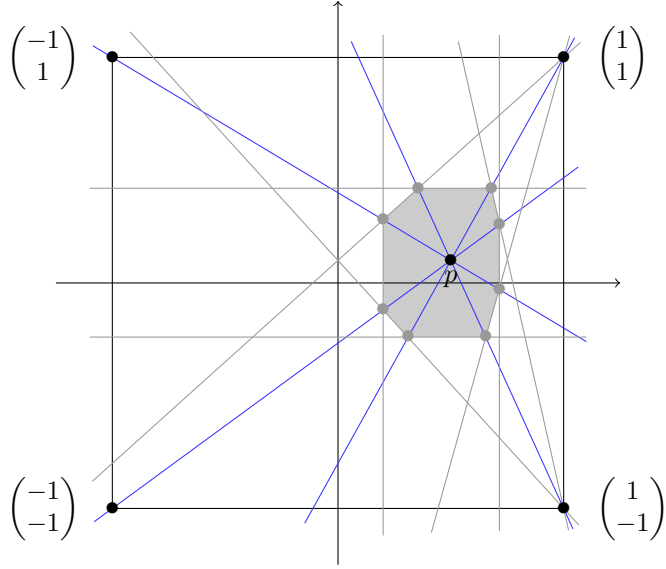
\begin{figure}[h]
\begin{tikzpicture}[scale=3]
\fill[black!20] (\z,{(\y+1)/(\x+1)*\z+(-1+(\y+1)/(\x+1))})--(\z,{(\y-1)/(\x+1)*\z+(1+(\y-1)/(\x+1))})-- (.355,.42)--(.68,.42)--(.715,.26)--(.715,-.03)--(.655,-.24)--(.31,-.24)--cycle;
\draw[->] (-1.25,0) -- (1.25,0);
\draw[->] (0,-1.25) -- (0,1.25);
\draw (-1,-1) -- (1,-1) -- (1,1) -- (-1,1)-- cycle;
\draw[blue!80,shorten >= -4cm,  shorten <=-0.30cm] (1,1) -- (\x,\y);
\draw[blue!80,shorten >= -3.2cm,  shorten <=-0.30cm] (1,-1) -- (\x,\y);
\draw[blue!80,shorten >= -2.1cm,  shorten <=-0.30cm] (-1,1) -- (\x,\y);
\draw[blue!80,shorten >= -2.1cm,  shorten <=-0.30cm] (-1,-1) -- (\x,\y);
\draw[black!40, shorten >= -5cm,  shorten <=-0.30cm] (1,-1) -- (\z,{(\y+1)/(\x+1)*\z+(-1+(\y+1)/(\x+1))});
\draw[black!40, shorten >= -3cm,  shorten <=-0.30cm] (\z,1)--(\z,{(\y+1)/(\x+1)*\z+(-1+(\y+1)/(\x+1)});
\draw[black!40, shorten >= -5.2cm,  shorten <=-0.30cm] (1,1)--(\z,{(\y-1)/(\x+1)*\z+(1+(\y-1)/(\x+1))}) ;
\draw[black!40, shorten >= -.3cm,  shorten <=-0.30cm] (1,-.24) -- (-1,-.24);
\draw[black!40, shorten >= -.3cm,  shorten <=-0.30cm] (1,.42) -- (-1,.42);
\draw[black!40, shorten >= -2cm,  shorten <=-0.30cm] (1,-1) -- (.68,.42);
\draw[black!40, shorten >= -2.7cm,  shorten <=-0.30cm] (1,1) -- (.655,-.24);
\draw[black!40, shorten >= -.3cm,  shorten <=-0.30cm] (.715, 1) -- (.715,-1);
\draw[black!40] (\z,{(\y+1)/(\x+1)*\z+(-1+(\y+1)/(\x+1))}) node{$\bullet$};
\draw[black!40] (\z,{(\y-1)/(\x+1)*\z+(1+(\y-1)/(\x+1))}) node{$\bullet$};
\draw[black!40] (.31,-.24) node{$\bullet$};
\draw[black!40] (.655,-.24) node{$\bullet$};
\draw[black!40] (.355,.42) node{$\bullet$};
\draw[black!40] (.68,.42) node{$\bullet$};
\draw[black!40] (.715,-.03) node{$\bullet$};
\draw[black!40] (.715,.26) node{$\bullet$};
\draw (\x,\y) node{$\bullet$};
\draw[below] (\x,\y) node{\small{$p$}};
\draw (1,1) node{$\bullet$};
\draw (1.1,1) node[right]{\small $\begin{pmatrix}
1\\1
\end{pmatrix}$};
\draw (1,-1) node{$\bullet$};
\draw (1.1,-1) node[right]{\small $\begin{pmatrix}
1\\-1
\end{pmatrix}$};
\draw (-1,-1) node{$\bullet$};
\draw (-1.1,-1) node[left]{\small $\begin{pmatrix}
-1\\-1
\end{pmatrix}$};
\draw (-1,1) node{$\bullet$};
\draw (-1.1,1) node[left]{\small $\begin{pmatrix}
-1\\1
\end{pmatrix}$};
\end{tikzpicture}
	\caption{Computation of the unit ball for the Hilbert metric in $Q_\st$ for a point $p$ in $Q_\st$.}\label{fig:UnitBallQst}
\end{figure}

We will have to subdivide our discussion into cases depending on the position of $p$ with respect to the lines $y=\pm x$. 
The first case is depicted in Figure~\ref{fig:UnitBallQst}.\\

\noindent\textbf{Case $\mathbf{|y|<x}$:} 
For $i=1,2,3,4$, let $v_{z_i}=(p,\xi_{z_i})$ be the unit tangent vector based at $p$ in the direction of $z_i$. For $z_i={\begin{pmatrix}1\\1\end{pmatrix}}$, there exists $\sigma_{z_1}\in\mathbb R$ such that
\[
\sigma_{z_1}\left(\begin{pmatrix}1\\1\end{pmatrix}-p\right)=(2-\sigma_{z_1})\left(p-\begin{pmatrix}\star\\-1\end{pmatrix}\right),
\] 
where the value of $\star$ is irrelevant to our purposes. From this we deduce that
\[
\xi_{z_1}=(1+y)\begin{pmatrix}1-x\\1-y\end{pmatrix}.
\]
Similar computations give
\[
\xi_{z_2}= (1-x)\begin{pmatrix}-(1+x)\\1-y\end{pmatrix}, \qquad 
\xi_{z_3}= -(1-x)\begin{pmatrix}1+x\\1+y\end{pmatrix}, \qquad 
\xi_{z_4}=-(1-y)\begin{pmatrix}-(1-x)\\1+y\end{pmatrix}.
\]
Since the unit ball in $T_pQ_\st$ is symmetric about the origin, these vertices and their opposites suffice to describe it. For the other cases, similar computations give the following:

\noindent\textbf{Case $\mathbf{|x|<y}$:}
\begin{gather*}
\xi_{z_1}=(1+x)\begin{pmatrix}1-x\\1-y\end{pmatrix},  
\qquad
\xi_{z_2}=(1-x)\begin{pmatrix}-(1+x)\\1-y\end{pmatrix} \\
\xi_{z_3}= -(1-y)\begin{pmatrix}1+x\\1+y\end{pmatrix}, 
\qquad 
\xi_{z_4}=-(1-y)\begin{pmatrix}-(1-x)\\1+y\end{pmatrix}.
\end{gather*}

\noindent\textbf{Case $\mathbf{x<-|y|}$:}
\begin{gather*}
\xi_{z_1}=(1+x)\begin{pmatrix}1-x\\1-y\end{pmatrix},  
\qquad
\xi_{z_2}= (1+y)\begin{pmatrix}-(1+x)\\1-y\end{pmatrix} \\
\xi_{z_3}= -(1-y)\begin{pmatrix}1+x\\1+y\end{pmatrix}, 
\qquad 
\xi_{z_4}=-(1+x)\begin{pmatrix}-(1-x)\\1+y\end{pmatrix}.
\end{gather*}

\noindent\textbf{Case $\mathbf{y<-|x|}$:}
\begin{gather*}
\xi_{z_1}=(1+y)\begin{pmatrix}1-x\\1-y\end{pmatrix},  
\qquad
\xi_{z_2}= (1+y)\begin{pmatrix}-(1+x) \\1-y\end{pmatrix} \\
\xi_{z_3}= -(1-x)\begin{pmatrix}1+x\\1+y\end{pmatrix}, 
\qquad 
\xi_{z_4}=-(1+x)\begin{pmatrix}-(1-x)\\1+y\end{pmatrix}.
\end{gather*}

\subsection{Dual unit ball}\label{sec: dual unit ball quad}
We can now prove Proposition~\ref{prop: quad area}.

\begin{proof}[Proof of Proposition~\ref{prop: quad area}]
We have seen in the last section how to compute the vertices of the unit ball at a point $p \in Q_\st\setminus \{y = \pm x\}$.
We would like to apply Theorem~\ref{thm:DualPolygon} to compute the dual of the symmetric polygon with vertices $\pm \xi_{z_1},\pm \xi_{z_2},  \pm \xi_{z_3}, \pm \xi_{z_4}$.
In all of the four cases the vertices 
\[
\xi_{z_1}, \xi_{z_2}, -\xi_{z_3}, \text{ and }-\xi_{z_4}
\] 
lie above the horizontal axis.
Recall that $\prec$ denotes the order on the first coordinates of vectors in $\bb R^2$.
To specify the order of the vertices, we need to treat the above four cases separately.
Recall that we have $-1<x,y<1$.\\

\noindent\textbf{Case $\mathbf{|y|<x}$:}
We have $
\xi_{z_2}\prec -\xi_{z_4}\prec \xi_{z_1}\prec -\xi_{z_3}$,  and thus the dual octagon has vertices $\pm w_1, \ldots, \pm w_4$ with
\begin{gather*}
w_1=\frac{1}{2}\begin{pmatrix}-\frac{1}{1-x^2}\\0\end{pmatrix},\qquad w_2=\frac{1}{4}\begin{pmatrix}-\frac{1}{1-x}\\\frac{1}{1-y}\end{pmatrix}, \qquad w_3=\frac{1}{2}\begin{pmatrix}0\\\frac{1}{1-y^2}\end{pmatrix},\qquad w_4=\frac{1}{4}\begin{pmatrix}\frac{1}{1-x}\\\frac{1}{1+y}\end{pmatrix};
\end{gather*}
see Theorem \ref{thm:DualPolygon}.
Again using the shoelace formula as in the proof of Lemma~\ref{lemma: integrand triangle},  we compute its (Lebesgue) area to be
\[
A_{Q_\st}(x,y)= \frac{2+x}{2(1-x^2)(1-y^2)}.
\]
We perform similar computations in the other three cases and obtain the following:\\

\noindent\textbf{Case $\mathbf{|x|<y}$:}
We have $\xi_{z_2}\prec -\xi_{z_4}\prec -\xi_{z_3}\prec \xi_{z_1}$ and $A_{Q_\st}(x,y)=\frac{2+y}{2(1-x^2)(1-y^2)}$.\\

\noindent\textbf{Case $\mathbf{x<-|y|}$:}
We have $ -\xi_{z_4}\prec\xi_{z_2}\prec -\xi_{z_3}\prec \xi_{z_1}$ and $A_{Q_\st}(x,y)= \frac{2-x}{2(1-x^2)(1-y^2)}$.\\

\noindent\textbf{Case $\mathbf{y<-|x|}$:}
We have $-\xi_{z_4}\prec\xi_{z_2}\prec  \xi_{z_1}\prec -\xi_{z_3}$ and $A_{Q_\st}(x,y)=\frac{2-y}{2(1-x^2)(1-y^2)}$.\\

Putting all cases together the area of the dual ball at $p=(x,y) \in Q_\st$ is
\[A_{Q_\st}(x,y)=\frac{2+\max\{|x|,|y|\}}{2(1-x^2)(1-y^2)},\]
and this concludes the proof of Proposition~\ref{prop: quad area}.
\end{proof}

\section{Example 1: Hyperbolic case}
\label{sec:FuchsianLocus}

The goal of this section is the proof of Proposition~\ref{prop: intro hyperbolic}.
For this assume that we have a positive quadruple of flags ${\bf F}=(E,F,G,H)$ with triple ratios 
\[
t=t'=T(E,F,G)=T(E,G,H)=1
\] 
and double ratios 
\[
d=d'=D_1(E,F,G,H)=D_2(E,F,G,H) \in (0,\infty). 
\]
In terms of geometric structures on surfaces, this corresponds to the case of having a quadruple of flags coming from a hyperbolic structure, seen as a convex real projective structure via the Klein--Beltrami model. For this reason, we say that ${\bf F}$ is {\em hyperbolic}. 
\begin{remark}\label{rmk:diagonal choice} Note that ${\bf F}$ being hyperbolic does not depend on the choice of oriented diagonal used to define the triple and double ratios of ${\bf F}$, as we now explain. First, observe that reversing the orientation on a diagonal in this case does not change the triple and double ratios, thanks to their symmetries described in Remark \ref{rmk: combinatorics triple and double}.  Furthermore, if we flip the diagonal from $E$ to $G$ to the diagonal from $F$ to $H$ as in Remark \ref{rmk: combinatorics triple and double}, the change of coordinates becomes
\begin{gather*}
T(F,E,H)=T(F,H,G)= \frac{1+d+d+d^2}{1+d+d+d^2}=1, \quad
D_1(F,E,H,G)=D_2(F,E,H,G)=\frac{1+d}{d(1+d)}=\frac{1}{d}.
\end{gather*}
\end{remark}
Under the assumption that the quadruple ${\bf F}$ is hyperbolic, the normalization from Section \ref{sec:HTQuadruple} (see also Figure \ref{fig:PosQuadruple}) becomes $\alpha=-\alpha_1=\alpha_2=\beta_1=-\beta_2 \in (-1,1)$ and $d=\tfrac{1+\alpha}{1-\alpha} \in (0,\infty)$. 
By symmetry we can assume that $\alpha\geq 0$.
Then, as in Figure \ref{fig: Fuchsian}, a point $p=(x,y)$ lies in the quadrilateral $Q$ if and only if
\begin{gather*}
|y +x| < 1 +\alpha \text{ and }|y -x| < 1 -\alpha.
\end{gather*}
Note that we know a priori that the area $\vol(\bf F)$ is finite as it less or equal to the Holmes--Thompson area of an ideal quadrilateral in the Klein--Beltrami model of the hyperbolic plane, which is $2\pi$ (see Remark~\ref{rem: volume decreasing}).
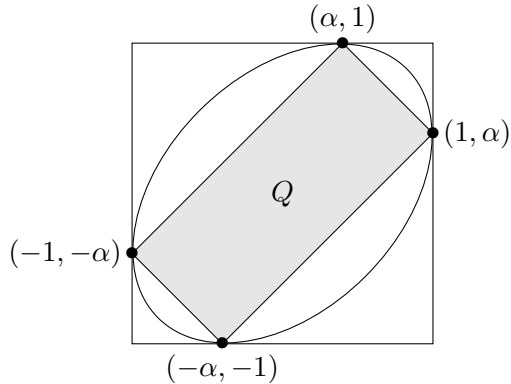
\begin{figure}[h]
\begin{tikzpicture}[scale=2]
\draw (-1,-1) -- (1,-1) -- (1,1) -- (-1,1)-- cycle;
\filldraw[fill=gray!20] (1,.4) -- (.4,1) -- (-1,-.4) -- (-.4, -1) --cycle;
\draw[rotate=45] (0,0) ellipse (1.17cm and .78cm);
\draw (1,.4) node{$\bullet$};
\draw (1,.4) node[right]{$(1,\alpha)$};
\draw (.4,1) node{$\bullet$};
\draw (.4,1) node[above]{$(\alpha,1)$};
\draw (-1,-.4) node{$\bullet$};
\draw (-1,-.4) node[left]{$(-1,-\alpha)$};
\draw (-.4,-1) node{$\bullet$};
\draw (-.4,-1) node[below]{$(-\alpha,-1)$};
\draw (0,0) node{$Q$};
\end{tikzpicture}
\caption{A pair of hyperbolic quadrilaterals}
\label{fig: Fuchsian}
\end{figure}

To compute the Holmes--Thompson area of $Q$ with respect to the standard quadrilateral $Q_{\st}$,  recall from Proposition~\ref{prop: quad area} that we need to integrate the function $A_{Q_\st}$ over $Q$.
Since in this subsection the outer quadrilateral is fixed to be $Q_\st$,  to improve readability we write $A$ instead of $A_{Q_\st}$.
We start by changing variables setting $x+y=u$ and $x-y=v$ and
using the identity $\max\{|x|,|y|\}=\tfrac{|x+y|}{2}+\tfrac{|x-y|}{2}$ to obtain
\begin{align}\label{eq:change of variables}
4 \int_{Q}A(x,y)dxdy&=2\int_{Q}\frac{2+\max\{|x|,|y|\}}{(1-x^2)(1-y^2)}dydx=\frac{1}{2}\int_{Q'}\frac{4+|u|+|v|}{(1-\tfrac{1}{4}(u+v)^2)(1-\tfrac{1}{4}(u-v)^2)}dudv,
\end{align}
where we integrate over the rectangle $Q'$ with vertices
\[
(1+\alpha,\pm(1-\alpha))\text{ and } (-(1+\alpha),\pm(1-\alpha)).
\]
Since the integrand is symmetric about the $u$- and $v$-axis,  it suffices to integrate over the positive quadrant.
Thus we obtain that 
\begin{align*}
4\int_{Q}A(x,y)dxdy &= 2\int_{0}^{1+\alpha}\int_{0}^{1-\alpha}\frac{4+u+v}{(1-\tfrac{1}{4}(u+v)^2)(1-\tfrac{1}{4}(u-v)^2)}dvdu\\
&=32\int_{0}^{1+\alpha}\int_{0}^{1-\alpha}\frac{4+u+v}{(4-(u+v)^2)(4-(u-v)^2)}dvdu
\end{align*}
Now, using partial fractions in the variable $v$, an elementary (but tedious) computation shows that
\begin{align*}
4\int_{Q}A&=\int_{0}^{1+\alpha}\frac{4}{u} \left(\frac{3 \ln (2-u-v)}{u-2}+\frac{(u+3) \ln
   (2+u-v)}{u+2}-\frac{(u+1) \ln (2-u+v)}{u-2}-\frac{\ln
   (2+u+v)}{u+2}\right)\Bigg\vert_{0}^{1-\alpha}du.
\end{align*}
Then, with the help of computational software \cite{Mathematica} and the equality
\[
\dil(x)-\dil(2-x)=\dil(x)+\dil(x-1)-\frac{\pi^2}{6}+\ln(x-1)\ln(2-x),  \ \text{ for }x<0,
\]
we deduce
\begin{align*}
4\int_QA=2 \Bigg[&3 \dil\left(-\frac{2-u}{1-\alpha}\right)+3
   \dil\left(-\frac{1+\alpha-u}{1-\alpha}\right)-\dil\left(\frac{u}{3-\alpha}\right)+\dil\left(-\frac{u}{3-\alpha}\right)-3
   \dil\left(-\frac{u}{1+\alpha}\right)\\
   &+3\dil\left(\frac{u}{1+\alpha}\right)-\dil\left(-\frac{2+u}{1-\alpha}\right)-\dil\left(-\frac{1+\alpha+u}{1-\alpha}\right)+3 \ln (1+\alpha-u) \ln
   \left(\frac{2-u}{1-\alpha}\right)\\
   &+\ln (1-\alpha) \ln
   \left(\frac{(1-\alpha)^2 (2+u)}{(2-u)^3}\right)-\ln
   \left(\frac{2+u}{1-\alpha}\right) \ln (1+\alpha+u)\\
   &+2
   \dil\left(-\frac{u}{2}\right)-2
   \dil\left(\frac{u}{2}\right)-\frac{\pi
   ^2}{3}\Bigg]\Bigg\vert_{0}^{1+\alpha}.
\end{align*}
Note that once the above expression for the integral is found, we can easily verify its correctness by computing its derivative.
Taking limits as $u \to 0$ and $u \to 1+\alpha$ is straightforward except (maybe) for the term $\ln (1+\alpha-u) \ln\left(\frac{2-u}{1-\alpha}\right)$, but we can check that
\[
\lim_{u\to 1+
\alpha}\ln (1+\alpha-u) \ln\left(\frac{2-u}{1-\alpha}\right)= 0.
\]
Thus we obtain
\begin{align*}
4\int_QA =&-2 \dil\left(-2\ \frac{1+\alpha}{1-\alpha}\right)+4
   \dil\left(\frac{1}{2} (-\alpha-1)\right)-4
   \dil\left(\frac{2}{\alpha-1}\right)-4\dil\left(\frac{\alpha+1}{2}\right)\\
   &-2\dil\left(\frac{\alpha+1}{3-\alpha}\right)+2\dil\left(\frac{\alpha+1}{\alpha-3}\right)-4
   \dil\left(\frac{\alpha+1}{\alpha-1}\right)-2
   \dil\left(\frac{\alpha+3}{\alpha-1}\right)-4 \ln
   \left(-\frac{2}{\alpha-1}\right) \ln (\alpha+1)\\
   &+2 \ln
   (1-\alpha) (-3 \ln (1-\alpha)+\ln (\alpha+3)+\ln (4))+2 \ln
   (2 (\alpha+1)) \ln \left(\frac{1-\alpha}{3+\alpha}\right)+\pi ^2\\
   =& \ 2 \ \Big(-\dil\left(-2d\right)
+2 \dil\left(\frac{-d}{1+d}\right)
-2\dil\left(-(1+d)\right)
-2\dil\left(\frac{d}{1+d}\right)\\
&-\dil\left(\frac{d}{2+d}\right)
+\dil\left(\frac{-d}{2+d}\right)
-2 \dil\left(-d\right)
-\dil\left(-(1+2d)\right)
-2\ln \left(1+d\right) \ln\left(\frac{2d}{1+d}\right)\\
&+\ln \left(\frac{2}{1+d}\right) 
(-3 \ln \left(\frac{2}{1+d}\right)
+\ln\left(\frac{2+4d}{1+d}\right)
+\ln (4))
+2 \ln \left(\frac{4d}{1+d}\right)
\ln \left(\frac{1}{1+2d}\right)
+\frac{1}{2}\pi ^2\Big),
\end{align*}
where we used $d=\frac{1+\alpha}{1-\alpha}$ in the last equality.
This proves Proposition~\ref{prop: intro hyperbolic}.

\begin{remark}
Using the formula above and computational software \cite{Mathematica}, one can show that
\[
\frac{d}{d\alpha}\Big\vert_{\alpha=0}\int_QA=0\qquad\text{and}\qquad\frac{d^2}{d\alpha^2}\Big\vert_{\alpha=0}\int_QA=\frac{32\ln(2)}{9}+8\ln(3)-16<0.
\]
Hence the Holmes--Thompson area has a local maximum at $\alpha=0$. 
One computes for $\alpha=0$,  that  
\begin{align*}
 \int_QA=& \frac{1}{4}\Big(-2 \dil\left(-2\right)+4 \dil\left(-\frac{1}{2}\right) -4\dil\left(-2\right)-4\dil\left(\frac{1}{2}\right)\\
   &-2\dil\left(\frac{1}{3}\right)+2\dil\left(-\frac{1}{3}\right)- 4\dil\left(-1\right)- 2\dil\left(-3\right)+2 \ln
   (2) \ln \left(\frac{1}{3}\right)+\pi ^2 \Big) \approx 4.66565,
\end{align*}
thus 
\[\vol(\bf{F})=\vol(E,F,G,H) =\frac{1}{\pi}\int_Q A \approx 1.48512.\]
In particular,  comparing to the area of an ideal quadrilateral in hyperbolic space,  we get
\[ \frac{\vol(E,F,G,H)}{2\pi} \approx 0.23636,\]
thus computing roughly $1/4$ of the hyperbolic area. 
Plotting the function $\vol({\bf F})$ with respect to the variable $\alpha\in (-1,1)$ suggests that $\alpha=0$ (equivalently,  $d=1$) is a global maximum \cite{Mathematica},  see Figure~\ref{fig:area_hyperbolic}.
\end{remark}
\begin{figure}[h!]
\includegraphics[scale=0.6]{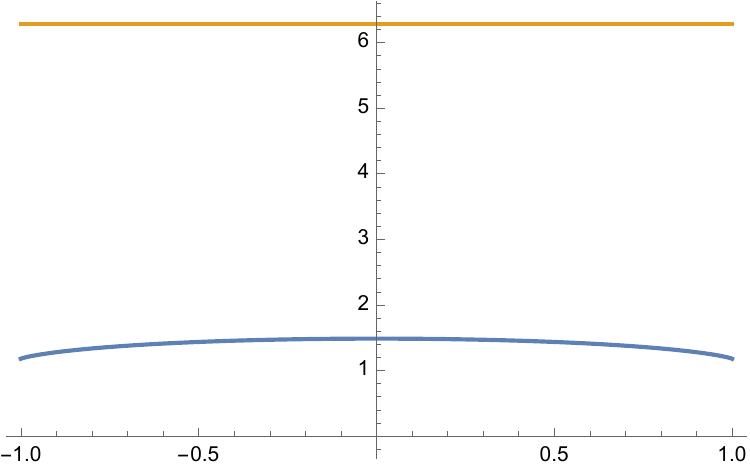}
\caption{The graph of the function $\vol({\bf F})$ in terms of the parameter $-1 < \alpha <1$. The horizontal line represents the hyperbolic area, which is constant equal to $2\pi$.}
\label{fig:area_hyperbolic}
\end{figure}

\begin{corollary}
We have $\lim_{\alpha \to 1} \frac{1}{\pi}\int_Q A= \frac{3}{8} \pi \approx 1.17810>0$. 
In particular, $\lim_{\alpha \to 1} \int_Q A \neq 0$.
\end{corollary}
\begin{proof}
We use the explicit calculation of $\int_Q A$ above and the following properties 
about the dilogarithm: $\dil(0)=0$, $\dil(1)=\frac{\pi^2}{6} $, $\dil (-1)=-\frac{\pi^2}{12}$,  and for negative real $z$,  one has
\[\dil(z) = -\frac{\pi^2}{6}-\frac{\ln^2(-z)}{2}-\dil\left(\frac{1}{z}\right),\]
see e.g.\ \cite[Sections 1 and 2]{Zagier_Dilog}.
Using this identity  we then obtain for some positive function $f(\alpha)$ satisfying $\lim_{\alpha\to 1} f(\alpha) >0$ (e.g.\ $f(\alpha) = \alpha+1$ or $f(\alpha)=\alpha+3$) that
\[ \lim_{\alpha \to 1} \dil\left( \frac{f(\alpha)}{\alpha-1}\right) =  \lim_{\alpha \to 1}  -\frac{\pi^2}{6}-\frac{1}{2} \ln^2\left(\frac{1-\alpha}{f(\alpha)}\right). \]
Applying this four times in the explicit expression of $\int_Q A$ and simplifying yields the desired result.
\end{proof}

\section{Example 2: $dd'=1$ and $tt'=1$}
\label{sec:HTQuadrupleEx2}
We are particularly interested in the example of a positive quadruple of flags $(E,F,G,H)$ such that
\begin{equation}\label{eq:tt=1dd=1}
tt'=T(E,F,G)\cdot T(E,G,H)=1\qquad\text{ and }\qquad dd'=D_1(E,F,G,H)\cdot D_2(E,F,G,H)=1,
\end{equation}
as these conditions appear when parametrizing finite area convex real projective structures on the thrice-punctured sphere, see Section~\ref{sec:ProjStrSphere} for details. 
Note, as in Remark~\ref{rmk:diagonal choice}, it follows from Remark~\ref{rmk: combinatorics triple and double} that the relations (\ref{eq:tt=1dd=1}) do not depend on the choice of oriented diagonal for the quadruple of flags $\mathbf{F}=(E,F,G,H)$, and thus we write $\bf{F}(d,t)$.

Using the normalized coordinates from Section \ref{sec:HTQuadruple} the equations (\ref{eq:tt=1dd=1}) become $\beta_1=\beta_2$ and $\alpha_1=\alpha_2$, thus we set $\alpha\coloneqq \alpha_1=\alpha_2$ and $\beta \coloneqq \beta_1=\beta_2$,  and we use the notation $Q=Q_{\alpha,\beta}$.
By symmetry of the area function $A=A_{Q_\st}$, we only need to consider the case $0\leq \alpha, \beta<1$. 
By work of Marquis \cite{marquis2012surface} (see the discussion in Section~\ref{sec:ProjStrSphere},  and Remarks~\ref{rem: volume bilipschitz} and~\ref{rem: volume decreasing}), we know a priori that the integral $\int_{Q_{\alpha,\beta}} A$ is finite.

\begin{figure}[h]
\begin{tikzpicture}[scale=2]
\draw (-1,-1) -- (1,-1) -- (1,1) -- (-1,1)-- cycle;
\filldraw[fill=gray!20] (1,.6) -- (.3,1) -- (-1,.6) -- (.3, -1) --cycle;
\draw (1,.6) node{$\bullet$};
\draw (1,.6) node[right]{$(1,\alpha)$};
\draw (.3,1) node{$\bullet$};
\draw (.3,1) node[above]{$(\beta,1)$};
\draw (-1,.6) node{$\bullet$};
\draw (-1,.6) node[left]{$(-1,\alpha)$};
\draw (.3,-1) node{$\bullet$};
\draw (.3,-1) node[below]{$(\beta,-1)$};
\draw (.2,0) node{$Q_{\alpha,\beta}$};
\end{tikzpicture}
\caption{We study the asymptotic behavior of the Holmes--Thompson area of $Q_{\alpha,\beta}$ as inscribed in the standard quadrilateral.}
\end{figure}
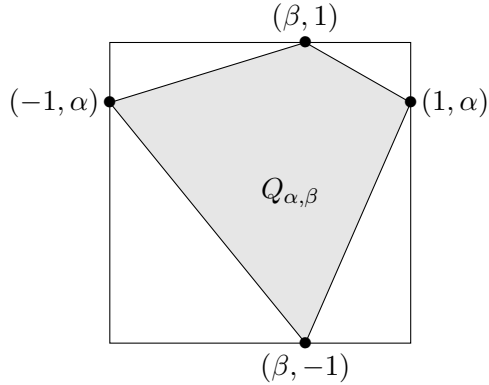

The main result of this section is as follows.
\begin{theorem}
\label{thm:mainS03} 
For $0 \leq \alpha, \beta<1$,  there exists a constant $C\in(0,\infty)$ such that
\[
\lim_{\max\{\alpha,\beta\}\to 1}\frac{\int_{Q_{\alpha,\beta}} A(x,y)dxdy}{\ln^2(1-\max\{\alpha,\beta\})}=C.
\]
\end{theorem}

By the equations~(\ref{eq: ai bi via triple and double}) in Section~\ref{sec:HTQuadruple},  we can express
\[\alpha=\frac{1-t}{1+t},  \quad \beta=-\frac{1-d}{1+d},\]
 and we note that $\alpha\geq 0$ if and only if $0<t\leq 1$,  and $\beta \geq 0$ if and only if $1 \leq d$.
Thus we obtain the following result as an immediate corollary.

\begin{corollary}[Theorem~\ref{thm: intro estimate}]
\label{corollary: estimate}
There exists a constant $C\in(0,\infty)$ such that
\[
\lim_{\|(\ln(d),\ln(t))\|\to\infty}\frac{\vol(\mathbf F(d,t))}{\ln^2(d)+\ln^2(t)}=C.
\]
\end{corollary}
\begin{proof}
We note that if $\alpha, \beta \geq0$,  then 
\[\ln^2(1-\alpha) = \ln^2\left(\frac{2t}{1+t}\right),  \; \textnormal{ and  } \ln^2(1-\beta) = \ln^2\left(\frac{2}{1+d}\right)= \ln^2\left(\frac{1+d}{2}\right),\]
which grow like $\ln^2(t)$ as $t\to 0$,  and $\ln^2(d)$ as $d \to +\infty$.
Thus as $d \to +\infty$ or $t\to 0$,  or equivalently $\beta \to 1$ or $\alpha \to 1$,  we can directly apply the result of Theorem~\ref{thm:mainS03}.

If now $1\leq t\to +\infty$,  then $-1 <\alpha\leq 0$,  and we need to replace $\alpha$ by $-\alpha$ in the computation of the area,  and we get by Theorem~\ref{thm:mainS03} that $\ln^2(1+\alpha) = \ln^2\big(\frac{2}{1+t}\big)=\ln^2\big(\frac{1+t}{2}\big)$, which grows like $\ln^2(t)$ as $t \to +\infty$.

Similarly,  if $1\geq d\to 0$,  then $-1 <\beta\leq 0$,  and we need to replace $\beta$ by $-\beta$,  and we get by Theorem~\ref{thm:mainS03} that $\ln^2(1+\beta) = \ln^2\big(\frac{2d}{1+d}\big)$, which grows like $\ln^2(d)$ as $d \to 0$.
Combining all the four possible cases completes the proof.
\end{proof}

\subsection{Outline of proof} We will break down the proof into several steps to improve readability.
The main work lies in the proof of Step 3.\\

\textbf{Step 1:} 
Let $T_\alpha$ (respectively $T_\beta$) denote the triangle contained in $Q_{\alpha,\beta}$ cut out by the lines $y=\pm x$ and $y=\alpha$ (respectively $x=\beta$),  see Figure~\ref{fig: Ta and Tb}. 
\begin{lemma}\label{lem:Talpha and Tbeta}
We have
\[
\lim_{\alpha\to 1}\frac{\int_{T_\alpha} A(x,y)dxdy}{\ln^2(1-\alpha)}=
\lim_{\beta\to 1}\frac{\int_{T_\beta} A(x,y)dxdy}{\ln^2(1-\beta)}=\frac{3}{8}.
\]
\end{lemma}
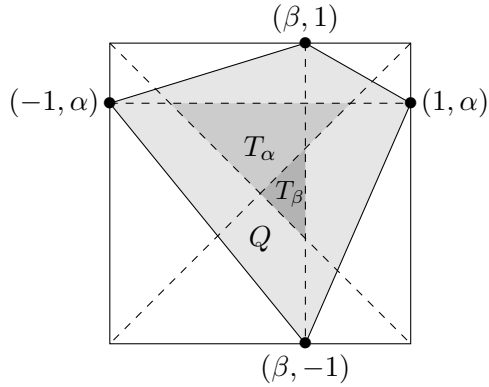
\begin{figure}[h]
\begin{tikzpicture}[scale=2]
\draw (-1,-1) -- (1,-1) -- (1,1) -- (-1,1)-- cycle;
\filldraw[fill=gray!20] (1,.6) -- (.3,1) -- (-1,.6) -- (.3, -1) --cycle;
\fill[fill=gray!40] (0,0)-- (.6,.6) -- (-.6,.6) --cycle;
\fill[fill=gray!60] (0,0) -- (.3,.3) -- (.3,-.3) --cycle;
\draw (1,.6) node{$\bullet$};
\draw (1,.6) node[right]{$(1,\alpha)$};
\draw (.3,1) node{$\bullet$};
\draw (.3,1) node[above]{$(\beta,1)$};
\draw (-1,.6) node{$\bullet$};
\draw (-1,.6) node[left]{$(-1,\alpha)$};
\draw (.3,-1) node{$\bullet$};
\draw (.3,-1) node[below]{$(\beta,-1)$};
\draw (0,-.3) node{$Q$};
\draw[dashed] (-1,.6)--(1,.6);
\draw[dashed] (-1,1)--(1,-1);
\draw[dashed] (1,1)--(-1,-1);
\draw[dashed] (.3,1)--(.3,-1);
\draw (0,.3) node{$T_\alpha$};
\draw (.2,0) node{\small $T_\beta$};
\end{tikzpicture}
\caption{The triangles $T_\alpha$ and $T_\beta$ cut out by the lines $y=\pm x$ and $y=\alpha$, respectively $x=\beta$. }
\label{fig: Ta and Tb}
\end{figure}

\textbf{Step 2.} 
Let $(\gamma,\gamma)$ be the intersection of the line $y=x$ and the line passing through $(1,\alpha)$ and $(\beta,1)$.
We have
\[
\gamma=\gamma_{\alpha,\beta}=\frac{1-\alpha\beta}{2-\alpha-\beta}\geq\max\{\alpha,\beta\}.
\]
Denote by $Q'=Q_{\alpha,\beta}'$ the square with vertices $(\pm\gamma,\pm \gamma)$ and $(\pm\gamma,\mp\gamma)$ as in Figure~\ref{fig: Q' and Deltai}. 
\begin{lemma}\label{lem:Q'} 
We have
\[
\lim_{\max\{\alpha,\beta\}\to 1}\frac{\int_{Q_{\alpha,\beta}'} A(x,y)dxdy}{\ln^2(1-\max\{\alpha,\beta\})}=\frac{3}{2}.
\]
\end{lemma}
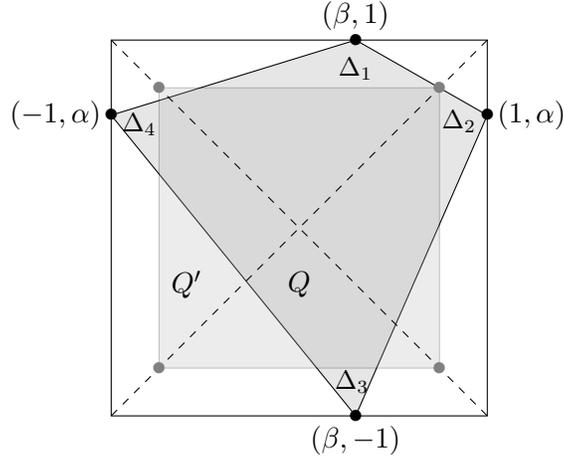
\begin{figure}[h]
\begin{tikzpicture}[scale=2.5]
\draw (-1,-1) -- (1,-1) -- (1,1) -- (-1,1)-- cycle;
\filldraw[fill=gray!20] (1,.6) -- (.3,1) -- (-1,.6) -- (.3, -1) --cycle;
\filldraw[fill=gray!60, nearly transparent] (0.74545454545,0.74545454545) -- (0.74545454545,-0.74545454545)--(-0.74545454545,-0.74545454545)--(-0.74545454545,0.74545454545) --cycle;
\draw (1,.6) node{$\bullet$};
\draw (1,.6) node[right]{$(1,\alpha)$};
\draw (.3,1) node{$\bullet$};
\draw (.3,1) node[above]{$(\beta,1)$};
\draw (-1,.6) node{$\bullet$};
\draw (-1,.6) node[left]{$(-1,\alpha)$};
\draw (.3,-1) node{$\bullet$};
\draw (.3,-1) node[below]{$(\beta,-1)$};
\draw (0,-.3) node{$Q$};
\draw (-.6,-.3) node{$Q'$};
\draw[dashed] (-1,1)--(1,-1);
\draw[dashed] (1,1)--(-1,-1);
\draw (0.74545454545,0.74545454545) node[gray]{$\bullet$};
\draw (-0.74545454545,0.74545454545) node[gray]{$\bullet$};
\draw (0.74545454545,-0.74545454545) node[gray]{$\bullet$};
\draw (-0.74545454545,-0.74545454545) node[gray]{$\bullet$};
\draw (.3,.85) node{\small $\Delta_1$};
\draw (-.85,.55) node{\small $\Delta_4$};
\draw (.28,-.82) node{\small $\Delta_3$};
\draw (.85,.57) node{\small $\Delta_2$};
\end{tikzpicture}
\caption{The square $Q'$ with vertices $(\pm\gamma,\pm \gamma)$ and $(\pm\gamma,\mp\gamma)$ and the connected components $\Delta_i$ of $Q \setminus Q'$. }\label{fig: Q' and Deltai}
\end{figure}

\textbf{Step 3.} Finally, denote by $\Delta_i=\Delta_{i,\alpha,\beta}$ for $i=1,2,3,4$ the connected components of $Q \setminus Q'$,  see Figure~\ref{fig: Q' and Deltai}. 
Note that these are four triangles since $\gamma\geq \max\{\alpha,\beta\}$.

\begin{proposition}\label{lem:Deltai} 
For all $i=1,2,3,4$ we have
\[
\lim_{\max\{\alpha,\beta\}\to 1}\frac{\int_{\Delta_i} A(x,y)dxdy}{\ln^2(1-\max\{\alpha,\beta\})}=0.
\]
\end{proposition}

\begin{proof}[Proof of Theorem \ref{thm:mainS03}] The proof follows at once by combining the statements of Lemmas \ref{lem:Talpha and Tbeta}, \ref{lem:Q'}, and Proposition \ref{lem:Deltai}.
\end{proof}

\subsection{Proof of Lemma \ref{lem:Talpha and Tbeta}} 
Recall the region $T_\alpha$ cut out by the lines $y=\pm x$ and $y=\alpha$.
Note that in the region $T_\alpha$ we have $y\geq |x|$,  and thus, up to a factor of 2, our goal is to compute 
\[
\int_{T_\alpha}\frac{2+y}{(1-y^2)(1-x^2)}dxdy.
\]
Parametrizing $T_{\alpha}=\{(x,y) \in Q_\st \mid -y\leq x\leq y,\ 0\leq y\leq \alpha\}$, we have
\[
\int_0^\alpha \frac{2+y}{1-y^2}\left(\int_{-y}^y\frac{1}{1-x^2}dx \right)dy=\frac{1}{2}\int_0^\alpha\frac{2+y}{1-y^2}\Big[\ln\left(\frac{1+x}{1-x}\right)\Big]\Big\vert_{-y}^ydy=\int_0^\alpha\frac{2+y}{1-y^2}\ln\left(\frac{1+y}{1-y}\right)dy \eqqcolon f(\alpha).
\]
We can compute,  again noting that $\ln\big(\frac{1+y}{1-y}\big)' = \frac{2}{1-y^2}$ and doing an integration by parts, that
\begin{align*}
f(\alpha)&=\left[\frac{2+y}{4}\ln^2\left(\frac{1+y}{1-y}\right)\right]\Bigg|_{0}^\alpha - \frac{1}{4} \int_0^\alpha \ln^2\left(\frac{1+y}{1-y}\right)dy\\
&= \left[\frac{2+y}{4}\ln^2\left(\frac{1+y}{1-y}\right)\right]\Bigg|_{0}^\alpha - \frac{1}{4} \int_0^\alpha\left( \ln^2(1+y)+\ln^2(1-y)-2\ln(1+y)\ln(1-y)\right)dy.
\end{align*}
The integrals of $\ln^2(1+y)$ and $\ln^2(1-y)$ can be computed by first doing a change of variables setting $x=1+y$ respectively $x=1-y$,  and then doing an integration by parts using the constant function $1$.
The integral of $\ln(1+y)\ln(1-y)$ is slightly more delicate,  but can be solved similarly: first we do a change of variables setting $x=1+y$,  then we do an integration by parts integrating $\ln(x)$.
Left is then to compute an integral of the form $\int \frac{x}{2-x}\ln(x)dx$. 
For this we first set $u =2-x$ to simplify further and to obtain an integral of the form $\int \frac{2}{u}\ln(2-u)du$.  
Now setting $v=\frac{u}{2}$ and recalling that $\dil(z)=-\int_0^z \frac{1}{t} \ln(1-t) dt$ solves also this last integral.
Putting everything together,  this gives
\begin{align*}
f(\alpha)
&=\dil\left(\frac{1-\alpha}{2}\right)-\frac{\pi^2}{12}+\frac{\ln^2(2)}{2}\\
&\qquad+\frac{1}{4} \Big(\ln ^2(1+\alpha)-4 \ln (2) \ln\left(1-\alpha\right)+ 3\ln^2 (1-\alpha)-2 \ln(1-\alpha)\ln (1+\alpha)\Big).
\end{align*}
Using that $\dil(0)=0$,  one directly sees that
\[
\lim_{\alpha\to 1}\frac{f(\alpha)}{\ln^2(1-\alpha)}=\frac{3}{4}.
\]
Now, note that in the region $T_\beta$ we have $x\geq |y|$,  and thus we want to compute
\[
\int_{T_\beta}\frac{2+x}{(1-y^2)(1-x^2)}dxdy.
\]
We can parametrize $T_{\beta}=\{(x,y)\in Q_\st \mid -x\leq y\leq x,\ 0\leq x\leq \beta\}$, so the same computation as above gives 
\[
\int_{T_\beta}\frac{2+x}{(1-y^2)(1-x^2)}dxdy \eqqcolon f(\beta),  \textnormal{ and }
\lim_{\beta\to 1}\frac{f(\beta)}{\ln^2(1-\beta)}=\frac{3}{4}.
\]

\subsection{Proof of Lemma \ref{lem:Q'}} Note that by symmetry and up to a factor of 2, it suffices to compute
\[
\int_{T_\gamma}\frac{2+y}{(1-y^2)(1-x^2)}dxy,
\]
where $T_\gamma$ is the region cut out by the lines $y=\pm x$ and $y=\gamma=\frac{1-\alpha\beta}{2-\alpha-\beta}$.
\begin{figure}[h]
\begin{tikzpicture}[scale=2]
\draw (-1,-1) -- (1,-1) -- (1,1) -- (-1,1)-- cycle;
\filldraw[fill=gray!20] (1,.6) -- (.3,1) -- (-1,.6) -- (.3, -1) --cycle;
\fill[fill=gray!40] (0,0)-- (0.74545454545,0.74545454545) -- (-0.74545454545,0.74545454545) --cycle;
\draw (1,.6) node{$\bullet$};
\draw (1,.6) node[right]{$(1,\alpha)$};
\draw (.3,1) node{$\bullet$};
\draw (.3,1) node[above]{$(\beta,1)$};
\draw (-1,.6) node{$\bullet$};
\draw (-1,.6) node[left]{$(-1,\alpha)$};
\draw (.3,-1) node{$\bullet$};
\draw (.3,-1) node[below]{$(\beta,-1)$};
\draw (0,-.3) node{$Q$};
\draw[dashed] (-1,0.74545454545)--(1,0.74545454545);
\draw[dashed] (-1,1)--(1,-1);
\draw[dashed] (1,1)--(-1,-1);
\draw (0,.3) node{$T_\gamma$};
\end{tikzpicture}
\caption{The triangle $T_\gamma$ cut out by the lines $y=\pm x$ and $y=\gamma$. }
\end{figure}
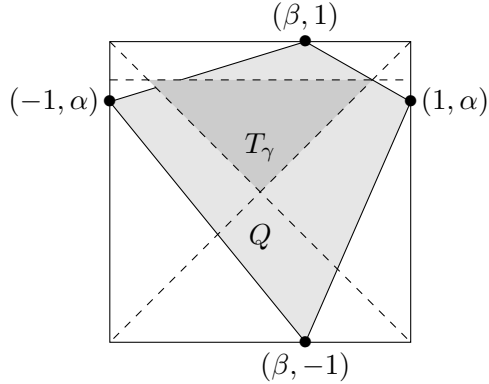

We can parametrize $T_{\gamma}=\{(x,y) \in Q_\st\mid -y\leq x\leq y,\ 0\leq x\leq \gamma\}$, so the same computation as in the proof of Lemma~\ref{lem:Talpha and Tbeta} gives 
\[
\int_{T_\gamma}\frac{2+y}{(1-y^2)(1-x^2)}dxdy \eqqcolon f(\gamma).
\]
We wish to compute
\[
\lim_{\max\{\alpha,\beta\}\to 1}\frac{f(\gamma)}{\ln^2(1-\max\{\alpha,\beta\})}.
\]
Since $f(\gamma)$ grows like $\ln^2(1-\gamma)$, it suffices to notice that 
\[
\lim_{\alpha\to 1}\frac{\ln(1-\gamma)}{\ln(1-\alpha)}=\lim_{\beta\to 1}\frac{\ln(1-\gamma)}{\ln(1-\beta)}=1.
\]

\subsection{Proof of Proposition \ref{lem:Deltai}}
Finally, note that $Q \setminus Q'$ consists of four disjoint triangles $\Delta_1,\dots,\Delta_4$ (we are omitting the dependence of these regions on $\alpha$ and $\beta$), each with a vertex on an edge of $Q$ and a side parallel to a coordinate axis.

\begin{proposition} For every $i=1,2,3,4$ we have
\[
\lim_{\alpha\to 1}\frac{\int_{\Delta_i} A(x,y)dxdy}{\ln^2(1-\alpha)}=\lim_{\beta\to 1}\frac{\int_{\Delta_i} A(x,y)dxdy}{\ln^2(1-\beta)}=0.
\]
\end{proposition}

The idea is to divide every $\Delta_i$ further into two adjacent triangles,  $\Delta^r_i$ and $\Delta^l_i$.
This is illustrated in Figure~\ref{fig: Delta left and right} for $\Delta_1$.
\begin{figure}[h]
\begin{tikzpicture}[scale=3]
\draw (-1,-1) -- (1,-1) -- (1,1) -- (-1,1)-- cycle;
\filldraw[fill=gray!20] (1,.6) -- (.3,1) -- (-1,.6) -- (.3, -1) --cycle;
\filldraw[fill=gray!60, nearly transparent] (0.74545454545,0.74545454545) -- (0.74545454545,-0.74545454545)--(-0.74545454545,-0.74545454545)--(-0.74545454545,0.74545454545) --cycle;
\filldraw[fill=gray!60] (.3,1) -- (.3,0.74545454545) -- (0.74545454545,0.74545454545) -- cycle;
\filldraw[fill=gray!60] (.3,1) -- (.3,0.74545454545) -- (-0.52727272728,0.74545454545) -- cycle;
\draw (1,.6) node{$\bullet$};
\draw (1,.6) node[right]{$(1,\alpha)$};
\draw (.3,1) node{$\bullet$};
\draw (.3,1) node[above]{$(\beta,1)$};
\draw (-1,.6) node{$\bullet$};
\draw (-1,.6) node[left]{$(-1,\alpha)$};
\draw (.3,-1) node{$\bullet$};
\draw (.3,-1) node[below]{$(\beta,-1)$};
\draw (0,-.3) node{$Q$};
\draw (-.6,-.3) node{$Q'$};
\draw[dashed] (-1,1)--(1,-1);
\draw[dashed] (1,1)--(-1,-1);
\draw (0.74545454545,0.74545454545) node[gray]{$\bullet$};
\draw (-0.74545454545,0.74545454545) node[gray]{$\bullet$};
\draw (0.74545454545,-0.74545454545) node[gray]{$\bullet$};
\draw (-0.74545454545,-0.74545454545) node[gray]{$\bullet$};
\draw (.3,0.74545454545) node{$\bullet$};
\draw (0.74545454545,0.74545454545) node{$\bullet$};
\draw (-0.52727272728,0.74545454545) node{$\bullet$};
\draw (.4,.83) node{$\Delta_1^r$};
\draw (.3,.65) node{$(\beta,\gamma)$};
\draw (.85,.85) node{$(\gamma,\gamma)$};
\draw (.1,.83) node{$\Delta_1^l$};
\end{tikzpicture}
\caption{The triangles $\Delta_1^r$ and $\Delta_1^l$ contained in $\Delta_1$. }
\label{fig: Delta left and right}
\end{figure}
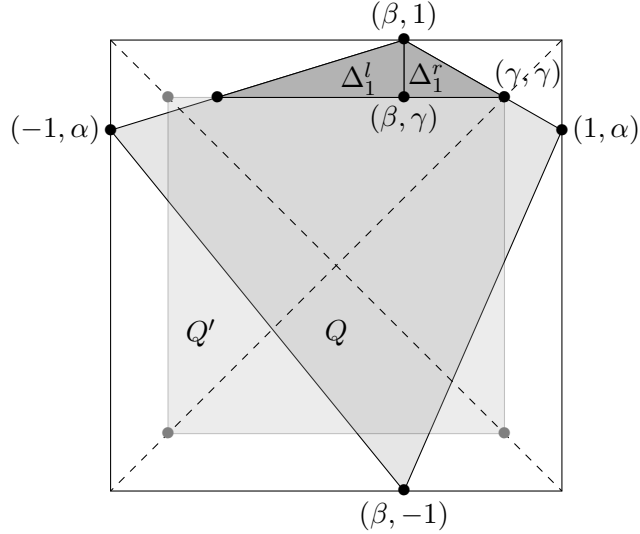
We start by considering the triangle $\Delta^r_1$ contained in $\Delta_1$ with vertices $(\beta,1)$, $(\beta,\gamma)$, and $(\gamma,\gamma)$. In particular, $y\geq |x|$ for all $(x,y)\in \Delta^r_1$.
Note that the line passing through $(\beta,1)$ and $(\gamma, \gamma)$ has equation 
\[
x=\beta+\frac{1-\beta}{1-\alpha}(1-y).
\]
\begin{proposition}
We have
\[
\lim_{\alpha\to 1}\frac{\int_{\Delta^r_1}\frac{2+y}{(1-x^2)(1-y^2)}dxdy}{\ln^2(1-\alpha)}=\lim_{\beta\to 1}\frac{\int_{\Delta^r_1}\frac{2+y}{(1-x^2)(1-y^2)}dxdy}{\ln^2(1-\beta)}=0.
\]
\end{proposition}
\begin{proof} 
Since $\Delta_1^r\subset Q$, we know that its Holmes--Thompson area is finite and we can apply the Fubini--Tonelli theorem by parametrizing
$\Delta_1^r=\{(x,y) \in Q_\st\mid  \gamma\leq y\leq 1,  \ \beta \leq x\leq \beta+\tfrac{1-\beta}{1-\alpha}(1-y)\}$ to write
\begin{align*}
\int_{\Delta^r_1}\frac{2+y}{(1-x^2)(1-y^2)}dxdy&=\int_{\gamma}^1\left(\frac{2+y}{1+y}\cdot \frac{1}{1-y}\left(\int_{\beta}^{\beta+\frac{1-\beta}{1-\alpha}(1-y)}\frac{dx}{1-x^2}\right)\right) dy\\
&<2\int_{\gamma}^1\frac{1}{1-y}\arctanh(x)\Bigg\vert_{\beta}^{\beta+\frac{1-\beta}{1-\alpha}(1-y)}dy\\
&=2\int_{\gamma}^1\frac{1}{1-y}\left(\arctanh\Big(\beta+\frac{1-\beta}{1-\alpha}(1-y)\Big)-\arctanh(\beta)\right)dy
\end{align*}
Set
\[
\delta\coloneqq\displaystyle\frac{1-\beta}{1-\alpha}>0\qquad\text{ and }\qquad F(y)\coloneqq\arctanh\left(\beta+\delta(1-y)\right)-\arctanh(\beta).
\]
Before we continue we need an auxiliary lemma.
\begin{lemma} For all $y\in[\gamma,1]$, we have
\[
F(y)\leq \frac{\arctanh(\gamma)-\arctanh(\beta)}{1-\gamma}(1-y).
\]
\end{lemma}
\begin{proof} We start by showing that $F(y)\geq 0$ for all $y\in[\gamma,1]$. Since $F(1)=0$, it suffices to show that $F$ is decreasing on this interval. Given that
\[
F'(y)=-\frac{\delta}{1-\left(\beta+\delta(1-y)\right)^2},
\]
the function $F$ is decreasing on the interval $[\gamma,1]$, if $0<\beta+\delta(1-y)<1$ for all $y\in[\gamma,1]$. This inequality can be observed geometrically since $\beta+\delta(1-y)$ gives the $x$-coordinate of a point on the line segment connecting $(\beta,1)$ to $(\gamma,\gamma)$ and $0<\beta<\gamma<1$. Now, it suffices to show that the function $F$ is concave up since
\[
x=\frac{\arctanh(\gamma)-\arctanh(\beta)}{1-\gamma}(1-y)
\]
is the equation of the secant between $(\gamma, F(\gamma))$ and $(1, F(1))$. It is straightforward to see that
\[
F''(y)=\frac{2\delta^2\left(\beta+\delta(1-y)\right)}{\left(1-\left(\beta+\delta(1-y)\right)^2\right)^2}\geq 0,
\]
which concludes the proof of the lemma.
\end{proof}
Applying the lemma, we estimate
\begin{align*}
\int_{\gamma}^1\frac{1}{1-y}\left(\arctanh\Big(\beta+\frac{1-\beta}{1-\alpha}(1-y)\Big)-\arctanh(\beta)\right)dy&<(\arctanh(\gamma)-\arctanh(\beta)),
\end{align*}
and observe that
\[
\lim_{\alpha\to 1}\frac{\arctanh(\gamma)-\arctanh(\beta)}{\ln^2(1-\alpha)}=\lim_{\beta\to 1}\frac{\arctanh(\gamma)-\arctanh(\beta)}{\ln^2(1-\beta)}=0,
\]
which completes the proof.
\end{proof}

Now, consider the triangle $\Delta_1^\ell$ with vertices $(\beta,1)$, $(\beta,\gamma)$, and the intersection between the horizontal line $y=\gamma$ and the line between $(\beta,1)$ and $(-1,\alpha)$, see Figure~\ref{fig: Delta left and right}. In particular, $y\geq |x|$ for all $(x,y)\in \Delta^\ell_1$.
Note that the line passing through $(\beta,1)$ and $(-1,\alpha)$ is given by the equation 
\[
x=\beta-\frac{1+\beta}{1-\alpha}(1-y).
\]
\begin{proposition} We have
\[
\lim_{\alpha\to 1}\frac{\int_{\Delta^\ell_1}\frac{2+y}{(1-x^2)(1-y^2)}dxdy}{\ln^2(1-\alpha)}=\lim_{\beta\to 1}\frac{\int_{\Delta^\ell_1}\frac{2+y}{(1-x^2)(1-y^2)}dxdy}{\ln^2(1-\beta)}=0.
\]
\end{proposition}
\begin{proof}
Again, since we know that the Holmes--Thompson area of $\Delta^\ell_1$ is finite, we can apply the Fubini--Tonelli theorem parametrizing $\Delta_1^\ell = \{(x,y) \in Q_\st \mid \gamma \leq y \leq 1,  \ \beta-\tfrac{1+\beta}{1-\alpha}(1-y) \leq x \leq \beta\}$ to write
\begin{align*}
\int_{\Delta^\ell_1}\frac{2+y}{(1-x^2)(1-y^2)}dxdy
&<2\int_{\gamma}^1\frac{1}{1-y}\arctanh(x)\Bigg\vert_{\beta-\frac{1+\beta}{1-\alpha}(1-y)}^{\beta}dy\\
&=2\int_{\gamma}^1\frac{1}{1-y}\left(\arctanh\left(\beta\right)-\arctanh\Big(\beta-\frac{1+\beta}{1-\alpha}(1-y)\Big)\right)dy.
\end{align*}
Observe that
\begin{align*}
2\Bigg(\arctanh&(\beta)-\arctanh\Big(\beta-\frac{1+\beta}{1-\alpha}(1-y)\Big)\Bigg)
=\ln\left(\frac{1+\beta}{1-\beta}\right)-\ln\left(\frac{1+\beta-\frac{1+\beta}{1-\alpha}(1-y)}{1-\beta+\frac{1+\beta}{1-\alpha}(1-y)}\right)\\
&=\ln\left(\frac{1+\frac{1+\beta}{1-\beta}\frac{1-y}{1-\alpha}}{1-\frac{1-y}{1-\alpha}}\right)=\ln\left(1+\frac{\frac{1-y}{1-\alpha}}{1-\frac{1-y}{1-\alpha}}\left(1+\frac{1+\beta}{1-\beta}\right)\right)=\ln\left(1+\frac{2}{1-\beta}\frac{\frac{1-y}{1-\alpha}}{1-\frac{1-y}{1-\alpha}}\right).
\end{align*}
Note that $1-\beta>0$ and $0<(1-y)/(1-\alpha)<1$, since $0<\alpha<y<1$. Thus we can use the inequality
\[
\ln(1+z)\leq z\text{ for all }z\geq 0
\]
to get
\begin{align*}
\int_{\Delta^\ell_1}\frac{2+y}{(1-x^2)(1-y^2)}dxdy
&<\int_{\gamma}^1\frac{1}{1-y}\frac{2}{1-\beta}\frac{\frac{1-y}{1-\alpha}}{1-\frac{1-y}{1-\alpha}}dy\\
&=\frac{2}{1-\beta}\int_\gamma^1\frac{1}{y-\alpha}dy=\frac{2}{1-\beta}\big(\ln(1-\alpha)-\ln(\gamma-\alpha)\big).
\end{align*}
Since
\[
\gamma-\alpha=\frac{1-\alpha\beta}{2-\alpha-\beta}-\alpha=\frac{1-\alpha\beta-2\alpha+\alpha^2+\alpha\beta}{2-\alpha-\beta}=\frac{(1-\alpha)^2}{2-\alpha-\beta},
\]
we conclude
\[
0<\int_{\Delta^\ell_1}\frac{2+y}{(1-x^2)(1-y^2)}dxdy<\frac{2}{1-\beta}\big(\ln(2-\alpha-\beta)-\ln(1-\alpha)\big).
\]
It is straightforward to check that
\[
\lim_{\alpha\to 1}\frac{\frac{2}{1-\beta}\big(\ln(2-\alpha-\beta)-\ln(1-\alpha)\big)}{\ln^2(1-\alpha)}=\lim_{\beta\to 1}\frac{\frac{2}{1-\beta}\big(\ln(2-\alpha-\beta)-\ln(1-\alpha)\big)}{\ln^2(1-\beta)}=0,
\]
which concludes the proof.
\end{proof}
Let $\vol(\Delta_i)$ for $i=1,2,3,4$ denote the Holmes--Thompson area of $\Delta_i$ as inscribed in the standard quadrilateral. By symmetry (reflection about the line $y=x$), we immediately deduce that if $\Delta_2$ denotes the connected component of $Q \setminus Q'$ with vertex $(1,\alpha)$, then
\[
\lim_{\alpha\to 1}\frac{\vol(\Delta_2)}{\ln^2(1-\alpha)}=\lim_{\beta\to 1}\frac{\vol(\Delta_2)}{\ln^2(1-\beta)}=0.
\]
Finally, if $R_x$ denotes the Euclidean reflection about the $x$-axis and $R_y$ denotes the Euclidean reflection about the $y$-axis, up to relabeling $\Delta_3$ and $\Delta_4$, since $\alpha,\beta\geq 0$, we have 
\[
\Delta_3\subseteq R_x(\Delta_1)\qquad \text{ and }\qquad \Delta_4\subseteq R_y(\Delta_2).
\]
Then, by symmetry
\[
\lim_{\alpha\to 1}\frac{\vol(R_x(\Delta_1))}{\ln^2(1-\alpha)}=\lim_{\beta\to 1}\frac{\vol(R_x(\Delta_1))}{\ln^2(1-\beta)}=0
\]
and
\[
\lim_{\alpha\to 1}\frac{\vol(R_y(\Delta_2))}{\ln^2(1-\alpha)}=\lim_{\beta\to 1}\frac{\vol(R_y(\Delta_2))}{\ln^2(1-\beta)}=0.
\]
Combining everything, this concludes the proof of Proposition \ref{lem:Deltai}. 

\section{Convex projective structures on a thrice-punctured sphere}
\label{sec:ProjStrSphere}

We apply the results of Section \ref{sec:HTQuadrupleEx2} to convex real projective structures on a thrice-punctured sphere $S=S_{0,3}$. 
The unique complete finite area hyperbolic structure on $S$ gives a conjugacy class of discrete and faithful representations $\rho_0' \colon \pi_1(S_{0,3}) \to \PGL_2(\mathbb R)$. 
Post-composing by the irreducible representation of $\PGL_2(\mathbb R)$ into $\PGL_3(\mathbb R)$, we get a conjugacy class $[\rho_0]$ of representations for which the holonomy around each puncture is conjugated to a maximal Jordan block.
Marquis \cite{marquis2012surface} showed that continuous deformations of $[\rho_0]$ in the $\PGL_3(\mathbb R)$-character variety such that the holonomy around each puncture remains conjugated to a maximal Jordan block correspond to convex real projective structures on $S$ with finite {\em Busemann} area. 
Given that the Busemann and Holmes--Thompson areas are bilipschitz equivalent,  see Remark~\ref{rem: volume bilipschitz}, the Holmes--Thompson area of these convex real projective structures is also finite. 
Moreover, these representations are positive in the sense of Fock and Goncharov \cite{FockGoncharov_ModuliSpacesConvexProjectiveStructuresSurfaces,  FockGoncharov_ModuliSpacesLocalSystemsHigherTeichmuellerTheory}. 

The space of finite area marked convex real projective structures on $S=S_{0,3}$ is homeomorphic to $\mathbb R^2$ and we briefly recall its Fock--Goncharov parametrization. 
Let $p_1,p_2,p_3$ denote the three punctures of $S$ and let $\mathcal T_{\textnormal{bal}}$ denote the balanced ideal triangulation of $S$ consisting of edges oriented from $p_i$ to $p_{i+1}$, with $i\in\mathbb Z/3\mathbb Z$. Given a finite area convex real projective structure $\rho$ on $S$, let $\wt {\mathcal T}_{\textnormal{bal}}$ denote the lift of the balanced ideal triangulation to the properly convex domain $\Omega_\rho \subset \mathbb{RP}^2$ and choose an ideal triangle $T_0\in \wt {\mathcal T}_{\textnormal{bal}}$ with vertices $\wt p_1,\wt p_2,\wt p_3$ appearing in this cyclic order around $\partial\Omega_\rho$. By considering the three ideal triangles in $\wt {\mathcal T}_{\textnormal{bal}}$ adjacent to $T_0$, we obtain points $\wt q_1,\wt q_2,\wt q_3\in\partial\Omega_\rho$ so that $\wt p_1,\wt q_3,\wt p_2,\wt q_1,\wt p_3,\wt q_2$ appear in this cyclic order around $\partial \Omega_\rho$. Note that there is a unique line tangent to $\partial\Omega_\rho$ at each of these points, and then denote by $(\xi_1,\xi_2,\dots, \xi_6)$ the positive 6-tuple of flags corresponding to $(\wt p_1,\wt q_3,\dots, \wt q_2)$. This positive 6-tuple of flags determines the following Fock--Goncharov parameters
\begin{gather*}
r_1=D_1(\xi_1,\xi_2,\xi_3,\xi_5),\qquad r_2=D_2(\xi_1,\xi_2,\xi_3,\xi_5),\\
b_1=D_1(\xi_3,\xi_4,\xi_5,\xi_1),\qquad b_2=D_2(\xi_3,\xi_4,\xi_5,\xi_1),\\
g_1=D_1(\xi_5,\xi_6,\xi_1,\xi_3),\qquad g_2=D_2(\xi_5,\xi_6,\xi_1,\xi_3),\\
t_1=T(\xi_1,\xi_3,\xi_5), \text{ and }
t_2=T(\xi_1,\xi_2,\xi_3)=T(\xi_3,\xi_4,\xi_5)=T(\xi_1,\xi_5,\xi_6).
\end{gather*}
Since we assume that $\rho$ has finite area, the Bonahon--Dreyer length equalities \cite[Proposition 13]{BonahonDreyer_ParametrizingHitchinComponents} imply that the parameters must satisfy the equalities
\begin{gather*}
r_1g_2=b_1r_2=g_1b_2=1\quad \text{ and }\quad r_2g_1t_1t_2=b_1r_2t_1t_2=g_1b_2t_1t_2=1.
\end{gather*}
Then, we can see that the choice of a triple ratio and a double ratio uniquely determines the remaining parameters. For the sake of concreteness, set $t=t_1$ and $d=r_1$. It follows that, if $\mathcal {CP}(S_{0,3})$ denotes the space of marked finite area convex real projective structures on $S_{0,3}$, we have a bijection
\begin{align*}
\Phi_{\mathcal T_{\textnormal{bal}}}\colon \mathcal{CP}(S_{0,3})&\to \mathbb R^2\\ 
\quad [\rho]&\to ({\bf d},{\bf t}) \coloneqq (\ln d,\ln t).
\end{align*}
We denote the inverse of $\Phi_\mathcal{T_{\textnormal{bal}}}$ by $\Psi_\mathcal{T_{\textnormal{bal}}}$. 

Observe that a quadrilateral obtained by two adjacent triangles in $\wt {\mathcal T}_{\textnormal{bal}}$ determines a positive quadruple of flags ${\bf F}=(E,F,G,H)$ which satisfies 
\begin{align}\label{eq: coord S03}
T(E,F,G)\cdot T(E,G,H)=1\quad\text{ and }\quad
D_1(E,F,G,H)\cdot D_2(E,F,G,H)=1.
\end{align}
In other words, the quadruple ${\bf F}$ satisfies the constraints considered in Section~\ref{sec:HTQuadrupleEx2}. 
Note that the equations~(\ref{eq: coord S03}) still hold after a diagonal flip or a change in orientations of the edges of $\mathcal T_{\textnormal{bal}}$ by Remark~\ref{rmk: combinatorics triple and double}. 

Now, let $\vol(\Psi_{\mathcal T_{\textnormal{bal}}}({\bf d},{\bf t}))$ denote the Holmes--Thompson area of the convex real projective structure $\rho$ parametrized by $({\bf d},{\bf t})$. Observe that $P_{\inn}(\bf F)$ is an ideal quadrilateral inscribed in $\Omega_\rho$ and that $\Omega_\rho$ is contained in the quadrilateral $P_{\out}(\bf F)$. Then, Remark~\ref{rem: volume decreasing} immediately implies the following.

\begin{corollary}
Let $\mathcal{T}_{\textnormal{bal}}$ denote the balanced ideal triangulation of the thrice-punctured sphere $S_{0,3}$. For $({\bf d},{\bf t})\in \mathbb R^2$ consider the corresponding finite area convex projective structure $\Psi_{\mathcal T_{\textnormal{bal}}}({\bf d},{\bf t})$. 
Let $\bf F$ be a positive quadruple of flags in $\mathbb R^3$ corresponding to two adjacent triangles in $\widetilde{\mathcal T}_{\textnormal{bal}}$ of the properly convex domain associated to $\Psi_{\mathcal T_{\textnormal{bal}}}({\bf d},{\bf t})$. 
Then,
\[\vol(\Psi_{\mathcal{T}_{\textnormal{bal}}}({\bf d},{\bf t})) \geq \vol(\bf F)=\vol_{P_\out(\bf F)}(P_{\inn}({\bf F})).\]
\end{corollary}
Then, using the results of Section \ref{sec:HTQuadrupleEx2},  we obtain the following corollary.

\begin{corollary}[Theorem~\ref{thm: intro: S03}]
There exists a constant $C>0$ such that for any $({\bf d},{\bf t})\in\mathbb R^2$, it holds that
\[\liminf_{\|({\bf d},{\bf t})\| \to \infty} \frac{\vol(\Psi_{\mathcal{T}_{\textnormal{bal}}}({\bf d},{\bf t}))}{\bf d^2+\bf t^2}\geq C.\]
\end{corollary}
\begin{proof}
By the above corollary we know that
\[\vol(\Psi_{\mathcal{T}_{\textnormal{bal}}}({\bf d},{\bf t})) \geq \vol({\bf F})\]
for a positive quadruple of flags ${\bf F}$ with Fock--Goncharov parameters $t,t',d,d'>0$ satisfying $tt'=1$ and $dd'=1$. 
Then, in Section~\ref{sec:HTQuadrupleEx2}, we showed that $\vol ({\bf F}(d,t))$ is comparable to $\max\{\ln^2(d),\ln^2(t)\}$,  see Corollary~\ref{corollary: estimate}, which finishes the proof.
\end{proof}

\bibliographystyle{alpha}
\bibliography{HTquad.bib}
\end{document}